\documentclass[11pt]{amsart}

\usepackage[utf8]{inputenc} 
\usepackage[T1]{fontenc} 
\usepackage{geometry}
\usepackage[normalem]{ulem} 
\usepackage{amsmath,amssymb,amsthm} 
\usepackage[colorlinks, pdfpagelabels, pdfstartview = FitH, bookmarksopen=true, bookmarksnumbered = true, linkcolor = black,plainpages = false, hypertexnames = false, citecolor = black]{hyperref} 
\usepackage{graphicx} 
\usepackage{tikz} 
\usetikzlibrary{patterns}
\usepackage[showseconds=false]{datetime2} 

\usepackage{caption}
\usepackage{subcaption} 

\newcounter{MODUS}
\ifnum\theMODUS=1 \newcommand{\DETAILS}[1]{{\mygreen{IN DETAIL:
      $\langle\hspace{-0.3ex}\langle$#1$\rangle\hspace{-0.3ex}\rangle$}}} \else
\newcommand{\DETAILS}[1]{} \fi

\newcommand{\qqquad}[0]{\qquad\qquad}

\newcommand{\BS}[0]{\backslash}
\newcommand{\nt}[0]{\notag}
\newcommand{\FA}[1]{\text{$\forall #1$}}

\newcommand{\HW}[1]{} 

\renewcommand{\t}{\tilde}
\newcommand{\p}{\partial}
\renewcommand{\d}{\delta}

\newcommand{\SUS}{\subset}

\newenvironment{TC} {\left \{\begin{array}{ll}} {\end{array} \right.}

\newcounter{pcounter}

\renewcommand{\AA}{{\mathcal A}}

\newcommand{\DD}{{\mathcal D}}
\newcommand{\EE}{{\mathcal E}}

\newcommand{\PP}{{\mathcal P}}

\newcommand{\R}{\mathbb{R}}
\newcommand{\N}{\mathbb{N}}

\newcommand{\C}{\mathbb{C}}

\newcommand{\alp}{\alpha}
\newcommand{\bet}{\beta}
\newcommand{\gam}{\gamma}
\newcommand{\eps}{\epsilon}

\newcommand{\lam}{\lambda}
\renewcommand{\phi}{\varphi}
\newcommand{\ome}{\omega}
\newcommand{\sig}{\sigma}

\newcommand{\Gam}{\Gamma}

\newcommand{\Ome}{\Omega}

\newcommand{\nin}{\not\in}

\DeclareMathOperator*{\DIV}{div}

\DeclareMathOperator*{\dist}{dist}

\DeclareMathOperator*{\spt}{spt}

\DeclareMathOperator*{\grad}{grad} 

\def \wto{\rightharpoonup}

\def\Xint#1{\mathchoice
{\XXint\displaystyle\textstyle{#1}}%
{\XXint\textstyle\scriptstyle{#1}}%
{\XXint\scriptstyle\scriptscriptstyle{#1}}%
{\XXint\scriptscriptstyle\scriptscriptstyle{#1}}%
\!\int}
\def\XXint#1#2#3{{\setbox0=\hbox{$#1{#2#3}{\int}$}
\vcenter{\hbox{$#2#3$}}\kern-.5\wd0}}

\def\dashint{\Xint-}

\newcommand{\upref}[2]{\hspace{-0.8ex}\stackrel{\eqref{#1}}{#2}} 

\newcommand{\lupref}[2]{\hspace{0ex} \stackrel{\eqref{#1}}{#2}} 

\usepackage{color}
\definecolor{verylightblue}{rgb}{0.95, 0.95, 0.95}  
\definecolor{lightblue}{rgb}{0.7, 0.7, 1}

\definecolor{eqyellow}{rgb}{0.9375,0.8984,0.5469}
\definecolor{subeqyellow}{rgb}{1,0.9373,0.8353}

\definecolor{mygreen}{rgb}{0.3, 0.6, 0.3} 
\definecolor{verylightgreen}{rgb}{0.2, 0.8, 0.2} 

\definecolor{verydarkgreen}{rgb}{0, 0.2, 0}

\definecolor{darkgreen}{rgb}{0.85, 0.85, 0.85}  
\definecolor{mydarkgreen}{rgb}{0, 0.5, 0} 

\definecolor{camouflagegreen}{rgb}{0.47, 0.53, 0.42}

\definecolor{mybrown}{rgb}{0.85, 0.4, 0.3}
\definecolor{verylightbrown}{rgb}{0.98, 0.72, 0.58}
\definecolor{verydarkbrown}{rgb}{0.44, 0.26, 0.26}

\definecolor{orange}{rgb}{1, 0.5, 0}

\definecolor{BurntOrange}{rgb}{0.9,0.356,0.1}

\definecolor{burlywood}{rgb}{0.87, 0.72, 0.53}

\definecolor{amethyst}{rgb}{0.6,0.4,0.8}

\newcommand{\magenta}{\color{magenta}}

\definecolor{brightcerulean}{rgb}{0.11, 0.67, 0.84}

\definecolor{mydarkred}{rgb}{1,0.086,0.255}

\definecolor{RoseVYDP}{rgb}{0.84,0.086,0.255}

\definecolor{dgreen}{rgb}{0, 0.8, 0.5}     

\definecolor{CanaryBRT}{rgb}{1,0.76,0.26}

\definecolor{cyan}{rgb}{0, 1, 1}
\definecolor{verylightgray}{rgb}{0.95, 0.95, 0.95}
\definecolor{verylightgray}{rgb}{0.95, 0.95, 0.95}
\definecolor{verylightred}{rgb}{1, 0.8, 0.78}
\definecolor{verylightyellow}{rgb}{0.99, 0.98, 0.5}
\definecolor{lightgray}{rgb}{0.8, 0.8, 0.8}

\definecolor{cyanprocess}{rgb}{0.0, 0.72, 0.92}

\definecolor{darkcyan}{rgb}{0.0, 0.45, 0.95} 

\definecolor{darkelectricblue}{rgb}{0.33, 0.41, 0.47}

\definecolor{darkmidnightblue}{rgb}{0.0, 0.2, 0.4}

\definecolor{darkpowderblue}{rgb}{0.2, 0.2, 0.6}  

\definecolor{cadet}{rgb}{0.33, 0.41, 0.47}

\definecolor{bole}{rgb}{0.47, 0.27, 0.23}

\definecolor{browntraditional}{rgb}{0.59, 0.29, 0.0}

\definecolor{burgundy}{rgb}{0.5, 0.0, 0.13}

\newcommand{\black}{\color{black}}

\newcommand{\mygreen}{\color{mygreen}}

\newcommand{\ignore}[1]{{}}

\newcommand{\NN}[1]{\|#1\|}
\newcommand{\NNN}[2]{\|#1\|_{#2}}

\newcommand{\skp}[2]{\langle #1, #2 \rangle}

\newcommand{\NT}[1]{\|#1\|_{L^2}}

\newcommand{\NI}[1]{\|#1\|_{L^\infty}}
\newcommand{\NIL}[2]{\|#1\|_{L^\infty({#2})}}
\newcommand{\NP}[2]{\|#1\|_{L^{#2}}}
\newcommand{\NPL}[3]{\|#1\|_{L^{#2}(#3)}}

\newcommand{\cciL}[1]{C_c^{\infty}(#1)}

\makeatletter
\newcommand{\VEC}[2][r]{
  \gdef\@VORNE{1}
  \left(\hskip-\arraycolsep%
    \begin{array}{#1}\vekSp@lten{#2}\end{array}%
  \hskip-\arraycolsep\right)}
\def\vekSp@lten#1{\xvekSp@lten#1;vekL@stLine;}
\def\vekL@stLine{vekL@stLine}
\def\xvekSp@lten#1;{\def\temp{#1}%
  \ifx\temp\vekL@stLine
  \else
    \ifnum\@VORNE=1\gdef\@VORNE{0}
    \else\@arraycr\fi%
    #1%
    \expandafter\xvekSp@lten
  \fi}
\makeatother

\DeclareMathOperator{\sgn}{sgn} \newcommand{\A}{\mathbb{A}}

\renewcommand{\hat}{\widehat} 

\newcommand{\hk}[1]{{\black{#1}}} 
\newcommand{\hkone}[1]{{\black{#1}}}
\newcommand{\hktwo}[1]{{\black{#1}}}

\newcommand{\hkk}[1]{{\black{#1}}}

\newcommand{\dsone}[1]{{\black{#1}}} %
\newcommand{\dstwo}[1]{{\black{#1}}} %
\newcommand{\ds}[1]{{\black{#1}}} %
\newcommand{\dsc}[1]{{\magenta{$\langle\hspace{-0.3ex}\langle$#1$\rangle\hspace{-0.3ex}\rangle$}}}
 
\newcommand{\dsk}[1]{{\black{#1}}} %
\renewcommand{\Re}{{\mathrm{Re}}}

\newcounter{margcount} 
\newcommand{\Rn}{\R^n}
\DeclareMathOperator{\supp}{spt}
\renewcommand{\NN}{\mathcal N}

\newcommand{\dx}[1]{\; \mathrm{d} #1} \newcommand{\sd}{\ : \:}
\newcommand{\dif}{\mathrm{div}} 

\newcommand{\bv}{\mathrm{BV}}
\newcommand{\f}{ {}_2F_1}

\newcommand{\hhom}{\dot{H}^{\frac{1}{2}}}

\newtheorem{theorem}{Theorem}[section]
 \newtheorem{lemma}[theorem]{Lemma}
 \newtheorem{remark}[theorem]{Remark}
\newtheorem{proposition}[theorem]{Proposition}

\begin{document}

\date\today 

\title{Asymptotic shape of isolated magnetic domains}



\author{Hans Knüpfer}
\address{Institut für Angewandte Mathematik, \\ Universität Heidelberg \& IWR, \\ Im Neuenheimer Feld 205, Germany}
\author{Dominik Stantejsky}
\address{Centre de Mathématiques Appliquées, UMR CNRS 7641, \\ École Polytechnique, IP-Paris \\ 91128 Palaiseau Cedex, France}

\parskip 6pt

\maketitle

\begin{abstract}
  We investigate the energy of an isolated magnetized domain
  $\Ome \SUS \Rn$ for $n=2,3$. In non--dimensionalized variables, the energy given by
  \begin{align*}  %
    \EE(\Ome) \ = \ \int_{\Rn} |\nabla \chi_{\Ome}| \ \dx x + \int_{\R^n} |\nabla \ds{h_\Omega}|^2 \ dx
  \end{align*}
  penalizes the interfacial area of the domain as well as the energy of the
  corresponding magnetostatic field.  Here, the magnetostatic potential $\ds{h_\Omega}$
  is determined by $\Delta \ds{h_\Omega} = \p_1 \chi_\Ome$, corresponding to uniform
  magnetization within the domain. We consider the macroscopic regime
  $|\Ome| \to \infty$, in which we derive compactness and $\Gam$--limit
  \hk{which is formulated in terms of} the cross--section\hk{al} \dstwo{area} of
  the anisotropically rescaled configuration. \dstwo{We then give the solutions
    for the limit problems.}

  \medskip


    \medskip
  

\end{abstract}


  

\section{Introduction}
Ferromagnetic materials exhibit the formation of \textit{magnetic domains},
i.e.\ regions with almost uniform magnetization oriented along certain
crystalline directions.  In the prototypical case of a sample with initial
uniform magnetization, the application of an external magnetic field can induce
a phase transformation. The initial time of the phase transformation is
characterized by the nucleation and growth of small isolated magnetic domains of
the new phase (see \cite[Figure 1]{Murakami2009}). The shape and energetics of these
domains is essential to understand the phase transformation of the magnetic
material and its related hysteresis
\cite{Geiss1996,Pokhil1997,Fabian1999,Desimone2002,DeSimone2000,Strukov1998}. \hkk{An
  isolated magnetic domain can also be formed} by a ferrofluidic droplet under
the application of an external field. In this paper, we investigate the shape
and energetics of \hkk{isolated magnetic domains} from the view point of
calculus of variations without \hkk{relying on any specific ansatz
  function}. We note that the shape and properties of magnetic domains in
solids and fluids have been studied in the physical literature both
experimentally as well as numerically \cite{Clark2013,Banerjee2001,Roodan2020},
for applications see e.g. \cite{Ochonski1989,Szczech2015,
  Uhlmann2002,Brusentsov2001,Newbower1973,Hess2015,Li2017}. However, in these
works the energy of the optimal domain shape is determined either numerically or
within a certain class of ansatz functions.

\medskip
  
We assume that the magnetization $m_\Ome \in L^2(\Rn,\Rn)$ points in direction
$e_1$ within the magnetic domain $\Ome \SUS \Rn$ and vanishes outside, i.e.
$m_\Ome = e_1 \chi_\Ome$ where $\chi_\Ome$ is the characteristic function of
$\Ome$. \hkk{By} Maxwell's equations the induced magnetostatic field
$h_\Ome \in L^2(\Rn,\Rn)$ is then given by the unique (distributional) solution
of
\begin{align}
  h_\Ome \ := \ -\nabla (-\Delta_{\Rn})^{-1} (\DIV m_\Ome). \qquad\qquad %
\end{align}
We note \hkk{that }(up to a factor 2), the same stray field is created if the
magnetization is $e_1$ in $\Ome$ and $-e_1$ \hkone{else} which would model the
situation of a uniformly magnetized domain in a\hkk{n infinitely extended}
uniaxial magnetic material. \hkk{O}ur choice of magnetization \hkk{also
}describes the situation of a uniformly magnetized ferrofluidic droplet. \hk{We
  assume that an external field with relative strength $\lam > 0$ is applied in
  $e_1$--direction.}  Also taking into consideration an energy related to the
interface of the magnetic domain,
the energy of the single magnetic domain in a non--dimensionalized setting is
given by
\begin{align} \label{intro-energy} %
  E(\Ome) \ = \ \int_{\Rn} |\nabla \chi_\Ome| \dx x %
  + \int_{\R^n} |h_\Ome|^2 \ \dx x + \lambda \int_{\R^n} \chi_\Ome \dx x.
\end{align} 
We note that the \dsone{last term in \eqref{intro-energy}, the Zeeman
  term,} is a trivial contribution to the energy \hkk{ by our assumption that
  the volume of the magnetic domain is given by
\begin{align}\label{vol-constraint}
  \int_{\Rn} \chi_\Ome \dx x \ = \ \mu \, ,
\end{align}
for some \hkk{fixed} $\mu > 0$.} We note that in the above model neither
boundary effects, nor material irregularities which might lead to pinning of the
domain are included. \DETAILS{ chapter 3.2.2, p.101.  Volume exchange Stiffness
  Energy $A\int (\grad m)^2 \dx V$ with a temperature dependent material
  constant $A$, the so called exchange stiffness constant. It has units of
  $J/m$. Zero temperature value $A(0)=k_B T_c/a_L$, where $k_B$ is Boltzmann's
  constant, $T_c$ Curie point and $a_L$ lattice constant, p.105 anisotropy
  energy $K_{cl}$ with units of $J/m^3$. p.109 External field (Zeeman) energy
  $E_H=-J_s\int H_\text{ex}\cdot m \dx V$ $J_s=$ Saturation magnetization, units
  of $T=N/(Am)$ p.109f stray field energy
  $E_d = \frac{1}{2}\mu_0 \int_{\R^n} H_d^2\dx V = -\frac{1}{2}\int_\Omega
  H_d\cdot J \dx V$ where $\dif(H_d) = -\dif(J/\mu_0)$ $\mu_0$ Vacuum
  permeability has units of $H/m$ or $N/A^2$ $H_d$ $A/m$} We note that the
physically relevant case corresponds to three dimensional domains
$\Omega$. However, for the mathematical interest, we formulate the model and
give some results in other dimensions as well. \hkk{We refer to Remark
  \ref{rem-one} for more details on the origin and derivation of the energy.}

\medskip

In \cite{Knupfer2017} it has been shown that minimizers of our energy functional
\eqref{intro-energy} with volume constraint \eqref{vol-constraint} exist in all
dimensions $n$ and for all prescribed masses $\mu\geq 0$.  For $2\leq n\leq 7$,
\hkk{\hkk{for any local minimizer of \eqref{intro-energy} there is a regular
    representative $\Ome$ for its positivity set which is open, bounded and has
    smooth boundary. Furthermore, in this case the sets $\Ome$,
    $\R^n\setminus\Ome$ and $\partial\Ome$ are connected.  We note that the
    positivity set is only defined a.e. and the regular representative is
    obtained after modification by a zero set}}.  Also, \hkk{the }scaling of
the \hkk{ground state} energy and some qualitative properties of the shape have
been investigated in \cite{Knupfer2017} for $n=3$ and in \cite{Stantejsky2018}
for $n \geq 4$. \hkk{We note} that for small volume $\mu$, the energy
\eqref{intro-energy} is a perturbation of the perimeter energy \hkk{and
  minimizers} have approximately the shape and energy of a ball with mass $\mu$.
We consider the opposite regime of large volume $\mu \gg 1$ where the impact of
the magnetostatic energy is essential to determine the \hkk{minimal energy and
  optimal }shape of the domain.  In this regime, minimizers of the energy
\eqref{intro-energy} get more and more elongated \hkk{and assume} the shape of a
thin needle.  More precisely, the scaling of the ground state energy is given by
$\inf_{|\Ome|=\mu} \EE(\Ome) \sim \mu^{\frac{2n-1}{2n+1}}$ for $n \geq 3$ for
$\mu \geq 1$.  The scal\hkk{ing}wise upper bound is obtained by ellipsoids with
length of order $L \sim \mu^\frac{3}{2n+1}$ and radius of order
$R \sim \mu^\frac{2}{2n+1}$ (cf. \cite{Stantejsky2018}).  In this paper, we
rescale the variables accordingly to maintain a bounded configuration. For the
rescaled energy, we establish a compactness result (Theorem
\ref{thm-compactness}) and a limit energy in the framework of
$\Gam$--convergence (Theorem \ref{thm-gamma}). The limit energy is formulated in
terms of the horizontal cross--section \dstwo{area} of the rescaled domain. The
limit energy is local \dstwo{for $n=3$} and can be solved algebraically for both
dimensions $n=2,3$ (Theorem \ref{thm-limit}). Interestingly, for $n = 2$ the
solution of the rescaled limit model yields the cross--section \dstwo{area} of
an exact ellipse \hkk{and is close to an ellipse for $n = 3$}.

\medskip
 
\hkone{We give some remarks about the relation to the underlying physical
  models:}
\begin{remark}[\hkone{Relation to }Landau--Lifshitz energy] \label{rem-one} %
  The energy of a ferromagnetic body \hkone{$U \SUS \R^3$} is given by the
  Landau-Lifshitz energy \cite[Ch. 3.2]{HS-Book}
  \begin{align} \label{E-phys} %
    \hkone{\EE_{LL}[m] \ =} \ \int_{\hkone{U}} \left( A |\hkone{\nabla} m|^2 + K
    (1-m_1^2)^2 - J_s H_\mathrm{ext}\cdot m \right) \dx x +
    \frac{\mu_0}{2}\int_{\mathbb{R}^3} |H|^2 \dx x.
\end{align}
The terms in \eqref{E-phys} in order of appearance are \hkone{as follows: The
  \emph{exchange energy} favors alignment of neighboring spins. The
  \emph{anisotropy energy} favours an alignment with certain crystal lattice
  directions; in our case we consider a uniaxial material where the directions
  $\pm e_1$ are preferred. The \emph{external field energy} (also called Zeeman
  energy) describes the energy associated with an external magnetic field
  $H_\text{ext}$}. The constant $J_s$ is the so called saturation
magnetization. Finally, the last integral on the right hand side of
\eqref{E-phys} is the \emph{stray field energy} (or magnetostatic energy) where
\hk{$H$} is the \emph{demagnetization} or \emph{stray field energy} and $\mu_0$
is the vacuum permeability. In bulk samples one observes large magnetic
  domains with uniform magnetization, separated by sharp one--dimensional
  transition layers (so called Bloch walls) of thickness of order
  $\ell_{trans} := \sqrt{A/K}$ with rapid rotation of the magnetization
  \cite[Sec. 3.6.1]{HS-Book}.

  \medskip

  From $|m| = 1$ we get the well--known estimate
  $a |\nabla m|^2 + b (1-m_3^2) \geq 2\sqrt{ab} |\nabla m_3|$. Hence,
  \begin{align}
    \EE_{LL}[m] \ \geq \  \int_U^3  \big( 2 \sqrt{A K} |\nabla m_1| - J_s H_\mathrm{ext}\cdot m \big) \ dx  + \frac{\mu_0}{2}  \int_{\mathbb{R}^3} |H|^2 \dx x.
  \end{align}
  We consider an external field of the form
  $H_{\textrm ext}= - \lam 8 (AK)^{3/2} \mu_0^{-2} e_1$ and non--dimen-sionalize
  the model by rescaling length and energy in units of
  $\ell_{res} = 4 \mu_0^{-1} \sqrt{AK}$ and $8 (AK)^{3/2} \mu_0^{-2}$. With the
  sharp interface assumption $m \in \{ \pm e_1 \}$ we arrive at
  \eqref{intro-energy}. This sharp interface approximation is valid if the width
  of the interfacial layers in \eqref{E-phys} is much smaller than the domain
  size \cite{AnzellottiBaldoVisintin-1991}. Note that we consider the
  renormalized energy where the infinite contribution of the Zeeman energy
  outside the domain is neglected. In view of \eqref{def-R-L-3d} our results
  (for $n = 3$) are hence physically relevant \hkk{for}
  sufficiently large domain\hkk{s} in the sense that $\ell_{trans} \ %
  \ll \ (V/\ell_{res}^3)^{\frac 27} (\ln V/\ell_{res}^3)^{-\frac 17}$, where $V$
  is the volume of the domain.
\DETAILS{ We resclae length by $x_old = \ell x_{new}$. Hence,
    \begin{align}
      \EE_{LL}[m] \ \geq \ 2  \ell^2 \sqrt{A K}  \int_\R^3  \big( |\nabla m_1| - J_s H_\mathrm{ext}\cdot m \big) \ dx  + \frac{\mu_0 \ell^3}{2}  \int_{\mathbb{R}^3} H_\Omega^2 \dx x.
    \end{align}
    With the choice $\ell = \frac 4{\mu_0} \sqrt{AK}$ we then obtain
  \begin{align}
    E(m) \ := \ %
    \int_{\Rn} |\nabla m| \dx x %
    + \int_{\R^n} |h(m)|^2  \ \dx x \ %
    + \lam  m_1 \big) \ dx    \leq \ \frac{\mu_0^2}{8 (AK)^{\frac 32}} \EE_{LL}[m]  .
  \end{align}
  with $\lam := - \frac{h_{\textrm ext} \mu_0^2}{8 (AK)^{\frac 32}}$.  The above
  estimate is expected to be sharp when the width of the interfacial layers is
  much smaller than the domain size. The width of the transition layer in the
  initial model is $\sqrt{A/K}$. The volume in the initial model is
  $V := \ell^3 \mu$. The width of the needle domain for $n = 3$ in the initial
  variables hence is
  \begin{align}
    \t R \ %
    := \ R_\mu \ell \ %
    = \ \mu^{\frac 27} (\ln \mu)^{-\frac 17} \ell \ %
    = \ \big(\tfrac {V}{\ell^3} \big)^{\frac 27} \Big(\ln \big( \tfrac
    {V}{\ell^3} \big) \Big)^{-\frac 17} \ell.
  \end{align}

  Our results are for $n = 3$ are hence relevant in the regime
  \begin{align}
    K \mu_0  \ %
    \ll \ \Big(\frac {V\mu^3}{(AK)^{3/2}} \Big)^{\frac 27} \Big(\ln \Big( \frac
    {V \mu^3}{(AK)^{3/2}} \Big) \Big)^{-\frac 17} 
  \end{align}
  where $V$ is the volume of the magnetic domain. We note that the energy of the
  Bloch wall more precisely depends on the orientation of the transition layer
  which is not accounted for in the above calculation (the check comment
  above). } 
\end{remark}
\hkone{\begin{remark}[Relation to ferrofluidic droplets] %
    A ferrofluid is a colloidal liquid consisting of ferromagnetic particles
    suspended in a surrounding fluid \cite{Rosensweig-Book}. The ferromagnetic
    particles are coated with a surfactant which inhibits agglomeration.  These
    fluids are interesting for applications since they can be moved by the
    application of an external field in which case fascinating droplet shape
    patterns may occur. Correspondingly, the associated magnetic energy is given
    by the exchange energy, Zeeman energy and the magnetostatic energy as in
    \eqref{E-phys}. Furthermore, for the total energy of the system we assume an
    interfacial energy at the boundary of the ferrofluidic droplet. For our
    model to apply we assume a constant density of the magnetic dipoles within
    the ferrofluidic droplet which then leads to the energy
    \eqref{intro-energy}.
\end{remark}
}
\medskip

\textbf{Notation.}  We write $X \lesssim Y$ if there exists a universal constant
$C>0$ such that $X\leq C Y$. Analogously, we also use $\gtrsim$.  If a constant
depends on a parameter $p$, we write $C_p$. For $x\in\R^n$ we write
$x=(x_1,x')$, $x'=(x_2,...,x_n)$ and analogously $\nabla=(\partial_1,\nabla')$.
The volume of the $n$--dimensional unit ball is denoted by $\ome_n$. The vector
space of all functions of bounded variation is denoted $\bv$.  For the perimeter
of a Lebesgue measurable set $E\subset\R^n$, we write $\PP(E)$.  The set $E$ is
said to have finite perimeter if $\PP(E)<\infty$
i.e. $\chi_E\in\bv(\R^n,\{0,1\})$ where $\chi_E$ is the characteristic function
of $E$.  The Fourier Transform and its inverse are given by
\begin{align}
  \widehat{f}(\xi) \ := \ \frac 1{(2\pi)^{\frac n2}} \int_{\R^n} f(x) e^{-i x\cdot\xi} \dx x, \, \qquad\qquad %
  f(x) = \frac 1{(2\pi)^{\frac{n}{2}}} \int_{\R^n} \widehat f(\xi) e^{i x\cdot\xi} \dx\xi.
\end{align}
We recall that the homogeneous $H^{s}$--norm for functions $f : \R \to \R$
is given by
\begin{align}
  \NNN{f}{\dot H^{s}}^2 \ = \ \int_{\R} |\hkone{\xi_1}|^{2s} |\widehat{f}(\xi_1)|^2 \dx\xi_1.
\end{align}

\section{Setting and statement of results} \label{sec-results}

\textbf{Setting and rescaled model.} In the limit of large $\mu$ we expect
minimizers to have the \hkk{approximate} shape of elongated ellipsoids. It is
hence convenient to work in \dsone{anisotropically} rescaled variables. For
prescribed volume $\mu > 1$ (cf. \eqref{vol-constraint}) we define
\begin{align} \label{trafo-u} u(\tilde{x}) \ := \ \chi_\Ome \Big(\frac{\t
    x_1}{R_\mu}, \frac{\t x'}{L_\mu} \Big).
\end{align}
The parameters $R_{\mu}, L_{\mu} > 0$ are given
by
\begin{align} \label{def-R-L-2d} 
R_{\mu} &= \mu^{\frac 13} , &&L_\mu = \mu^{\frac 23} &&\text{for $n=2$,} \\ %
R_{\mu} &= \mu^{\frac 27}(\ln \mu)^{-\frac{1}{7}}, &&L_\mu = \mu^{\frac37}(\ln\mu)^{\frac{2}{7}}, &&\text{for $n=3$,} \label{def-R-L-3d} \\ %
R_{\mu} &= \mu^{\frac{2}{2n+1}} , &&L_\mu = \mu^{\frac{3}{2n+1}}, &&\text{for $n\geq 4$.} \label{def-R-L-nd}
\end{align}
These length scales correspond to the optimal scaling of diameter $R_\dsk{\mu}$ and
length $L_\dsk{\mu}$ for the \hkone{magnetic} domain which are obtained by minimizing
the energy in the class of ellipsoidal configurations with volume $\mu$. For the aspect ratio, we write
\begin{align} \label{def-eps} 
  \eps \ := \ \frac{R_\mu}{L_\mu} \ < \ 1\, .
\end{align}
We correspondingly rescale our energy by $R_{\mu}^{n-2}L_{\mu}$, according to
the energy of these ellipsoidal configurations.  With
the change of variables \eqref{trafo-u} and expressing the nonlocal part of the
energy in terms of Fourier variables, the renormalized and rescaled energy takes
the form
\begin{align*}
  E_\eps^{(n)}[u] \ %
  := \ \frac{E(\Ome) - \lambda |\Ome|}{R_{\mu}^{n-2}L_{\mu}} \  %
  = \ P_\eps^{(n)}[u] + N_\eps^{(n)}[u], 
\end{align*}
where the parts of the energy related to perimeter and nonlocal interaction are
\begin{align}
  P_\eps^{(n)}[u]  \ &:= \ \int_{\R^n} \sqrt{|\nabla' u|^2 + \eps^2 |\partial_1 u|^2 } \dx x, \label{def-Peps} \\
  N_\eps^{(n)}[u] \  &:= \ \gam_n(\eps) \int_{\R^n} \frac{{\xi}_1^2}{\eps^2 {\xi}_1^2 + |{\xi}'|^2} |\widehat{u}({\xi})|^2 \dx\xi. \label{def-Neps}
\end{align}
Here,  the prefactor in front of the nonlocal energy in \eqref{def-Neps} is given by
\begin{align}
  \gam_2(\eps) \ %
  := \ \eps, \qquad 
  \gam_3(\eps) \ := \ \frac{1}{7|\ln \eps| - 3 \ln(\ln \mu)} \qquad \text{and} \ \ \gam_n(\eps):=1 \text{ for } n\geq 4.
\end{align}
We note that the limit $\eps \to 0$ is equivalent to $\mu \to \infty$ and that
$\gam_3(\eps) = \frac 17 |\ln(\eps)|^{-1} + o(1)$ as $\eps \to 0$. In view of the
assumption $|\Ome| = \mu$ and since $R_\mu^{n-1}L = \mu$, the function $u$ has
unit mass. Correspondingly, we consider the set of admissible functions
\begin{align} \label{def-AA} %
  \AA \ = \ \left\{ u\in \bv(\R^n,\{0,1\}) \sd \int_{\R^n} u(x) \dx x \ = \ 1
  \right\}
\end{align}
and set $E_\eps^{(n)}[u] = \infty$ if $u \nin \AA$.

\medskip

\textbf{Main results.}  %
In \cite{Nolte2018,Knupfer2017,Stantejsky2018}, existence of minimizers with
prescribed volume has been established for the energy \eqref{intro-energy}.
We \hkk{recall} the scaling of the ground state energy:
\begin{theorem}[Scaling of minimal energy
  \hkk{\cite{Nolte2018,Knupfer2017,Stantejsky2018}}] \label{thm-scaling} %
  Let $n \geq 2$. Then we have
  \begin{align} \label{thm-scaling-eq_inf} %
    c_{n} \ \leq \ \inf_{u \in \AA} E_\eps^{(n)}[u]
    \ \leq \ C_{n} \, \qquad\qquad %
    \text{for all $\eps \in (0,1)$}
  \end{align}
  for some constants $c_{ n}$, $C_{ n}>0$ which only depend on $n$.
\end{theorem}
\begin{figure}
\centering
     \begin{subfigure}[c]{0.35\textwidth}
         \centering
\begin{tikzpicture}[scale=1.7]
\draw[->] (0,0) -- (0,0.2) node[above] {$x_1$};
\draw[->] (0,0) -- (0.2,0) node[right] {$x'$};
\draw[line width=1, ->] (0,1) -- (0,1.5) node[right] {$H_{\textrm{ext}}$};
\definecolor{colD}{RGB}{200,200,200}
\draw[colD, fill=colD] (1,0.0) arc [start angle=-80, delta angle=160, radius=0.5cm] arc [start angle=180-80, delta angle=160, radius=0.5cm];

\node[colD!40!black] at (1,0.5) {$m=e_1$};
\node[black] at (2,0.5) {$m=-e_1$};

\end{tikzpicture}
         \caption{}
         \label{fig:fig-domain-field-a}
    \end{subfigure}
	\begin{subfigure}[c]{0.3\textwidth}
         \centering
\begin{tikzpicture}[scale=1.7]
\definecolor{colD}{RGB}{200,200,200}
\draw[colD, fill=colD] (1,0.0) arc [start angle=-55, delta angle=110, radius=1.2cm] arc [start angle=180-55, delta angle=110, radius=1.2cm];
\node[colD!40!black] at (1,0.9) {$m=e_1$};
\node[black] at (2,0.9) {$m=-e_1$};
\end{tikzpicture} 
         \caption{}
         \label{fig:fig-domain-field-b}
     \end{subfigure}
     \begin{subfigure}[c]{0.3\textwidth}
         \centering
\begin{tikzpicture}[scale=1.7]
\definecolor{colD}{RGB}{200,200,200}
\draw[colD, fill=colD] (1.25,0.0) arc [start angle=-40, delta angle=80, radius=3.2cm] arc [start angle=180-40, delta angle=80, radius=3.2cm];
\node[colD!40!black] at (1.25,1.75) {$m=e_1$};
\node[black] at (2.5,1.75) {$m=-e_1$};
\end{tikzpicture}
\caption{}
         \label{fig:fig-domain-field-c}
     \end{subfigure}
     \caption{Schematic representation of needle domains $\Omega$ with
       increasing volume
       \eqref{fig:fig-domain-field-a}-\eqref{fig:fig-domain-field-c}. The vector
       on the left shows the direction of the applied external magnetic
       field. The ratio between length (along $x_1$) of the domain and its
       thickness increases with the volume.}
  \label{fig-domain-field}
\end{figure}
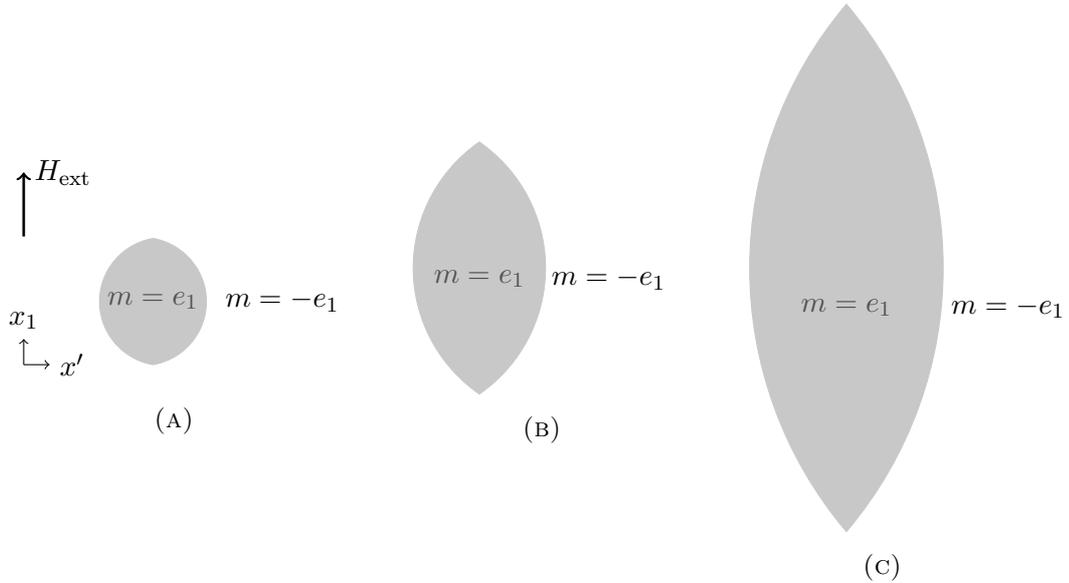
The proof of this statement is given in \cite{Nolte2018,Knupfer2017} for dimension $n=3$ and \cite{Stantejsky2018} for general $n\geq 3$ where the scaling laws \eqref{def-R-L-2d}-\eqref{def-R-L-nd} have been derived for suitable (integral) notions of diameter and length of the configuration.
Theorem \ref{thm-scaling} shows that our choice of $R_\mu$ and $L_\mu$ provides
us with a nontrivial limit as $\eps\rightarrow 0$.  In dimensions $2$ and $3$, under the assumption of uniformly bounded support of $u$, the statement of Theorem \ref{thm-scaling} \dsone{also follows} from Theorem \ref{thm-gamma} which additionally yields the asymptotically leading order constant in the expansion of the energy.  
Indeed, it is easy to see that by continuity of $\eps\mapsto\inf_{\AA} E_\eps^{(n)}$, the infimum of $E_\eps^{(n)}$ on $\AA$ 
\dstwo{is of order one} for $\eps\rightarrow 0$.  Taking an (almost)
minimizing sequence $u_\eps$ for $\eps\rightarrow 0$, we can apply
Theorem \ref{thm-gamma} to obtain the required bounds.  \dsone{In the case
  $n\geq 4$ we do not know the $\Gamma-$limit of the energy $E_\eps^{(n)}$
  but with the ideas presented in Lemma \ref{lem-app-scaling}, we are able to
  prove the lower bound in Theorem \ref{thm-scaling} with a different argument
  and we also obtain the upper bound in arbitrary dimension $n\geq 4$.  The main
  difference between $n=2,3$ and $n\geq 4$ is the property that
  $\gamma_n(\eps)\rightarrow 0$ (for $\eps\rightarrow 0$) only for
  $n=2,3$. This translates to the fact that in dimension $n\geq 4$ the leading
  order term does no longer consist of a single expression (cf. Proposition
  \ref{prp-upper} and Lemma \ref{lem-app-scaling}).  } 

\medskip

The limit problem is formulated in terms of the cross--sectional area
\begin{align} \label{def-cross} %
  \mathbb{A}[u] : \R \to [0,\infty), \qquad \mathbb{A}[u] \ := \ \int_{\R^{n-1}}
  u(\cdot,x') \dx x'.
\end{align}
In view of our unit mass assumption on $u$ in \eqref{def-AA}, we then have for $A:=\mathbb{A}[u]$ that
  \begin{align} \label{A-unit} %
    \int_\R A(x)\dx x \ = \ 1\, .
  \end{align}
  For \hkk{sequences} with uniformly bounded energy we have the following
  compactness \hkk{result:}
\begin{theorem}[Compactness] \label{thm-compactness} %
  Let $n \in \{ 2,3,4 \}$. Let $\rho > 0$ and let
    \begin{align}
      \text{$s \in \Big[0,\frac{1}{14}\dsk{\Big]}$ \ if $n = 2$, \quad %
      $s \in \Big[0,\frac 1{10}\Big)$ \ if $n = 3$ \quad and \quad %
      $s \in \Big[0,\frac{1}{22}\Big)$ \ if $n = 4$. }
    \end{align}
    Then for any sequence $u_\eps\in \AA$ with $E_\eps^{(n)}[u_\eps] \leq C$ and
    $\spt u_\eps \subset B_\rho(0)$ \FA{\eps > 0} we have
  \begin{align} \label{top-conv} %
    \A[u_\eps] \ \wto \ A \quad\text{ in } H^s(\R) \quad \text{as $\eps \to 0$}
  \end{align}
  (after selection of a subsequence) for some $A \ \in \ \AA_0^{(n)}$ where 
  \begin{align*}
    \AA_0^{(n)} \ := \ \Big\{ A\in X \sd  \text{$A$ satisfies \eqref{A-unit} }\Big\}
  \end{align*}
  and where $X = H^{\frac 12}(\R)$ if $n = 2$, $X = H^{1}(\R)$ if $n=3$ and
  $X = \bigcap_{s < \frac 7{10}}H^s(\R)$ if $n= 4$.
\end{theorem}
In particular, we have $\A[u_\eps] \to A$ in $L^2$ for all $2\leq n\leq 4$, see
\cite[Theorem 4.54]{Demengel2012}. We note that the regularity of the limit
space is better than the topology of the convergence in \eqref{top-conv}. To
have compactness in the topology of the limit space would require to rule out
oscillations on the frequency scale $\eps^{-1}$.  However, uniform control over
the energy does not seem to provide sufficient control in this regime. We also
note that the topology \eqref{top-conv} does not distinguish between two
configurations $u_{\eps,1}$ and $u_{\eps,2}$ if they have the same
cross--section \dstwo{area} and in particular is not \hk{even} Hausdorff.

\medskip

In the macroscopic limit $\eps \to 0$ and for dimensions $n \in \{ 2, 3 \}$, our
functionals $\Gam$--converge to \hktwo{limit} functionals given by
\begin{align}
  E_0^{(2)}[A] \ %
  &= \ 2\ \ds{|\{x\in\R\sd A(x)>0\}|} \ +\  \frac 12 \ \int_{\R} | |\partial_1|^{\frac 12} A(x_1)|^2 \dx x_1\, , \label{E0-2} \\ 
  E_0^{(3)}[A] \ %
  &= \ 2\sqrt{\pi} \int_\R \ \sqrt{A(x_1)}  \dx x_1 \ +\  \frac{1}{14\pi}\  \int_{\R} |\partial_1 A(x_1)|^2 \dx x_1, \label{E0-3}
\end{align}
for $A \in \AA_0^{(n)}$ and $E_0^{(n)}[A] = +\infty$ else. 
More precisely, we have:
\begin{theorem}[Gamma-limit] \label{thm-gamma} %
  Let $n \in \{ 2, 3 \}$. Then
  \begin{align}
    E_\eps^{(n)}  \ \stackrel{\Gam}{\longrightarrow} \ E_0^{(n)} \qquad\qquad
    \text{w.r.t to the topology \eqref{top-conv}. }
  \end{align}
In particular,
  \begin{enumerate}
  \item For any sequence $u_\eps\in \AA$ with uniformly bounded support such that \eqref{top-conv} holds for some $A \in \dsone{\AA_0^{(n)}}$, we have
    \begin{align} \label{gamma-liminf-eq} \liminf_{\eps\to 0}
      E_\eps^{(n)}[u_\eps] \ \geq \ E_0^{(n)}[A]. %
      \qquad\qqquad\text{(liminf estimate)}
    \end{align}
  \item For any $A\in \AA_0^{(n)}$, there is a sequence $u_\eps\in\AA$ with
    uniformly bounded support such that \eqref{top-conv} holds and
    \begin{align} \label{gamma-limsup-eq} \limsup_{\eps\to 0}
      E_\eps^{(n)}[u_\eps] \ \leq \ E_0^{(n)}[A]. %
      \qquad\qqquad\text{(limsup estimate)}
    \end{align}
  \end{enumerate}
\end{theorem}
The notion of $\Gam$--convergence in particular ensures that minimizers
\hkk{(with uniformly bounded support)} converge to minimizers of the limit
functional: Indeed, by \eqref{gamma-limsup-eq} any sequence of minimizers
$u_\eps$ of $E_\eps^{(n)}$ verifies the hypothesis of Theorem
\ref{thm-compactness} and by lower semicontinuity \eqref{gamma-liminf-eq} it
converges (up to extracting a subsequence) to a minimizer $A$ of the limit
functional $E_0^{(n)}$.  \hktwo{We note that we need the assumption of uniformly
  bounded suppport in Theorem \ref{thm-gamma} and hence only characterize
  minimizing sequences with this property. In fact, we believe that minimizing
  sequences should have equibounded support since the minimal energy of the
  limit functional is sublinear as a function of the volume. However, a proof
  for this statement seems quite technical and is \hkk{ongoing work}. We note
  that \cite[Thm. 1.2]{Knupfer2017} \hkk{gives a characterization of an}
  integral version of the length and radius of configuration.}
\begin{figure}[h]  \centering
\begin{tikzpicture}[scale=3]
\draw[->] (0,0) -- (0,1.1) node[above] {$t$};
\draw[] (0.05,1) -- (-0.05,1) node[left] {$1$};
\draw[] (1,0.05) -- (1,-0.05) node[below] {$1$};
\node at (-0.05,-0.05) {$0$};
\draw[->] (0,0) -- (1.1,0);
\draw[dotted] (0,0.5) -- (1.732/2,0.5) node[right] {$R^{(2)}(t)$};
\draw[red, line width=1] (0,1) arc (90:00:1);
\end{tikzpicture}
  \hspace{6ex}
\begin{tikzpicture}[scale=3]
\draw[->] (0,0) -- (0,1.1) node[above] {$t$};
\draw[] (0.05,1) -- (-0.05,1) node[left] {$1$};
\draw[] (1,0.05) -- (1,-0.05) node[below] {$1$};
\node at (-0.05,-0.05) {$0$};
\draw[->] (0,0) -- (1.1,0);
\draw[dotted] (0,0.5) -- (0.87,0.5) node[right] {$R^{(3)}(t)$};
\pgfmathsetmacro{\C}{7.753}
\pgfmathsetmacro{\lam}{2.435}
\pgfmathsetmacro{\L}{0.3212}
\pgfmathsetmacro{\R}{3.728}

\draw[samples=2000,domain=0:\R,smooth,variable=\rho,red, line width=1] plot ({\rho/\R}, {1 - \L*( 7/(2*sqrt(\lam)^3) * rad(atan(7/sqrt(4*\lam*\C)))   + sqrt(\C)/\lam - 7/(2*sqrt(\lam)^3) * rad(atan((7-2*\lam*\rho)/sqrt(4*\lam*(\C+7*\rho-\lam*\rho^2)))) - sqrt(\C+7*\rho-\lam*\rho^2)/\lam )});

\draw[blue, dashed] (0,1) arc (90:00:1);
\end{tikzpicture}
  \caption{Plot for the solutions $R^{(n)}(t)$ in Theorem \ref{thm-limit} for
    $0\leq t\leq 1$ \hk{with associated radially symmetric and rescaled needle
      domain $\Ome = \{ |x'| \leq \rho(x_1)\}$ }. \hk{For t}he solution of the
    limit problem \hk{$\Ome$} for $n=2$ is precisely the unit disk\hk{, for
      $n=3$ it approximates the unit disk (indicated by the dashed line)}.}
    \label{fig-min}
\end{figure}
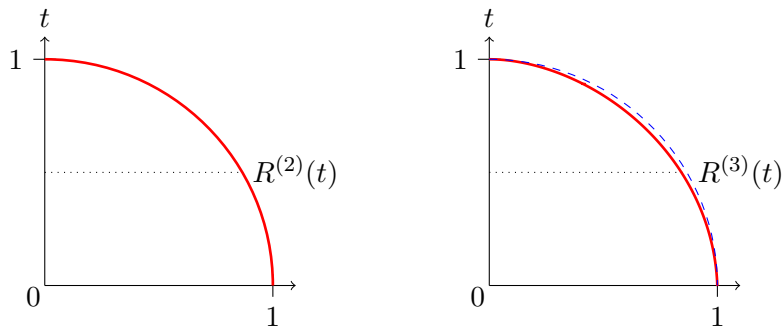
 
The limit energies which determine the asymptotic shape of the rescaled needles
can be solved explicitly (for $n=2$) or implicitly (for $n = 3$). An
  illustration of these solutions is shown in Fig. \ref{fig-min}.
\begin{theorem}[Solutions of limit problems] \label{thm-limit} %
  Let $n \in \{ 2, 3 \}$. The unique (up to translations) minimizers with bounded support
  $A^{(n)} \in \AA_0^{(n)}$ of $E_0^{(n)}$ are given by
  \begin{align}\label{thm-limit-def-A-R}
    A^{(n)}(x_1) \ = \ \ome_{n-1}  \big(R^{(n)}_* R^{(n)}(t)\big)^{n-1}\, , \qquad %
    \text{where \ $t \ := \ \frac{\hk{|x_1|}}{L_*^{(n)}}$}
  \end{align}
  and where $R_*^{(n)}$, $L_*^{(n)} > 0$ and $R^{(n)} : \R \to [0,1]$ are given as follows:
  \begin{enumerate}
  \item For $n = 2$, we have $R^{(2)}_* = (\frac 1{2\pi^2})^{\frac{1}{3}}$,
    $L_*^{(2)} = (\frac{2}{\pi})^{\frac 13}$ and
    \begin{align} \label{thm-limit-eq-2d} R^{(2)}(t) \ = \ \chi_{(-1,1)}(t) \
      \sqrt{1 - t^2}.
    \end{align}
  \item For $n=3$, \hk{$R^{(3)}$ is given by \eqref{lim-ELG-3D-impl-def-rho}  with parameters $R^{(3)}_*\approx 1.511$ and  $L^{(3)}_*\approx 0.202$ solving the system \eqref{lim-eq-abcd-2}-\eqref{lim-eq-abcd-4}.}
  \end{enumerate}
\end{theorem}

We remark that for $n = 2$, the minimizer of $E_0^{(2)}$ is exactly given by the
cross--section of an ellipse\hkk{; for $n = 3$ this is only approximately the
  case.  We believe that the optimal needle shape for the full model should be
  rotationally symmetric; however, to show this rigorously seems quite difficult
  and is outside of the scope of our analysis.  We have the following physical
  intuition for the limit model: } \hkk{With increasing volume the aspect ratio
  $R_\mu/L_\mu$ of the domain becomes smaller and the magnetostatic energy
  related to interaction between charges in horizontal direction becomes more
  important compared to interactions in other directions. In particular, the
  self--interaction of magnetic charges within a single layer
  $\{x_1\} \times \R^{n-1}$ for $x_1 \in \R$ becomes dominant in the limit. We
  note that $\p_1 A$ corresponds to the magnetic charge density per layer. A
  problem similar to the limit problem of Theorem \ref{thm-limit} \hkk{for
    $n = 3$} has been solved in \cite[Sec. 5]{KnuepferMuratov-2011}. Different
  than the problem considered in \cite{KnuepferMuratov-2011} we cannot assume an
  a priori bound on the support of our solution and we need to rule out the
  possibility of losing mass at infinity.  Also, our problem includes a volume
  constraint which is not present in \cite{KnuepferMuratov-2011}. }

\medskip

\textbf{Comparison to previous results.}
Nonlocal problems with an isoperimetric term and a competing nonlocal
interaction term have been studied extensively in recent years
\cite{Pegon2021,CandauTilh2021}, in particular \hkk{for} nuclear \hkk{stability}
\cite{Gamow1930}, elasticity and micromagnetics.  Coulomb-type potentials with
interfacial penalization and prescribed volume have been considered
(e.g.\cite{Knuepfer2012, Knuepfer2013, Lu2013, Julin2014, Bonacini2014,
  Bonacini2015, Frank2015,
  Alama2020}) since they naturally appear in physical models, for example in the
modeling of diblock-copolymers, see
\cite{Bahiana1990,Choksi2003,Leibler1980,Matsen1994,Ohta1986}. \hktwo{The phase
  transformation in a ferromagnetic slab under the influence of an external
  field has been investigated in \cite{KnuepferMuratov-2011} where also the
  nucleation of ferromagnetic domains \hkk{plays} an essential role, see also
  \cite{OttoViehmann-2009}.  Nucleation and related problems in shape memory
  alloys have been e.g. considered in
  \cite{KnuepferKohn-2010,ContiSchweizer-2006,KnuepferOtto-2019}. There are
  various other related models from materials science where dilute/large volume
  regimes are relevant, such as e.g. epitaxial growth
  (e.g. \cite{BellaGZ-2015}), dislocations (e.g. \cite{ContiGO-2015}) or
  superconductors (e.g. \cite{ChoksiCKO-2008,ContiGOS-2018}). }

\medskip

In many papers related to our setting, a perfect ellipsoid shape is assumed
\cite{Brown1968, Berkovsky1985, Brancher1987, Aharoni1988, Alouges2008,
  Fratta2016, Agarwal2011}.  This is a common simplification due to the fact
that the self-demagnetization of ellipsoids can be treated analytically.  From
this perspective, our 3D result which shows that asymptotic minimizers are not
perfect ellipsoids is of particular interest. We note however, that similar
structures have been observed numerically e.g.\ in
\cite{SeroGuillaume1992}. There, the authors study a sessile ferrofluid drop for
a specified range of surface tension, gravitational and magnetic forces. These
ranges are characterized by the magnetic and gravitational Bond numbers which
express the relative importance of gravitational and magnetic forces with
respect to surface tension.  Compared to our setting, gravitation (in the
vertical direction) is considered as an additional effect which favors flat
droplets.  The authors deduce numerically that if the gravitational Bond number
is below a certain threshold, then the height of the drop is increasing, seen as
function of the magnetic Bond number. If the magnetic Bond number is large
enough, an inversion of the curvature is observed, compare Figure \ref{fig-min}
and \cite[Fig. 4]{SeroGuillaume1992}. Other examples of optimal shapes for
ferromagnetic liquid droplets that are not ellipsoidal are given in
\cite{Banerjee2001}.  The energy considered there includes an exchange and core
energy in addition to surface and demagnetization energy in our paper.  The
authors constrain the set of possible shapes and textures by imposing a
cylindrical symmetry. They then use numerics to calculate the energy minimizers
in their restricted class of admissible shapes to be donut- or apple-shaped.

\medskip

The competition between a nonlocal repulsive potential and an attractive
confining term is found also in other problems, for example in models studying
the interaction of dislocations \cite{Meurs2021,Kimura2020} or
\cite{Mora2018,Carrillo2019,Carrillo2020}.  The latter considered a nonlocal
anisotropic potential energy with an isotropic and local confinement term, other
\cite{EleuteriLT-preprint2020} study problems with a nonlocal anisotropic
perimeter.  In our case, we combine a nonlocal anisotropic energy potential with
an anisotropic but local perimeter.  \hkk{An}other anisotropic and nonlocal
repulsive energy that has been treated variationally using ansatz--free analysis
is \cite{Carrillo2020} (based on \cite{Mora2018,Carrillo2019}).  The model
studied in \cite{Mora2018} describes the interaction of positive dislocations in
the plane. The energy and results have subsequently been generalized to
arbitrary dimensions in \cite{Carrillo2020} and consist of an isotropic
quadratic confinement and a\hkk{n} anisotropic, nonlocal and repulsive term with
a fixed mass constraint. The energy density of the latter expression
$W_\alpha(x)$ is associated with the isotropic Coulomb potential together with
an anisotropic part of the form $\alpha\frac{x_1^2}{|x|^n}$.  The authors prove,
that their energy admits a unique minimizer which is calculated to be an
ellipsoid with half axis depending on $\alpha$.  Their approach relies on the
strict convexity of the energy for $\alpha\in (-1,n-2]$ and the Euler-Lagrange
conditions. Borrowing and extending techniques from complex analysis, it is
shown that the optimality conditions are satisfied by ellipsoids. Although similar in
its statements, this result exhibits some major differences compared to our
work: Firstly, we note that our \hkk{(rescaled)} model incorporates anisotropy
in both parts of the energy. In particular, our perimeter bound degenerates for
$\eps\rightarrow 0$, while the attractive potential in \cite{Carrillo2020} is
purely isotropic.  Secondly, we point out that our repulsive energy potential is
"more degenerate" in the following sense: Writing $\NN_1$ in $x-$coordinates
(see Lemma 2.1 in \cite{Knupfer2017}) we find the convolution kernel to be
$\partial_1^2 \Gam_n(x) \sim \frac{-1}{|x|^n} + n\frac{x_1^2}{|x|^{n+2}}$, where
$\Gam_n$ is the usual Coulomb kernel in $n$ dimensions. Since the relevant
energy estimates take place close to $0$, the kernel $\partial_1^2 \Gam_n(x)$ is
more singular compared to $W_\alpha(x)$.  Thirdly, we want to highlight the
non-convexity of our problem. As we pointed out earlier, a crucial step in
\cite{Carrillo2020} is to use strict convexity for the existence of minimizers
and the conclusion that a solution to the Euler-Lagrange conditions is indeed a
global minimizer. Our energy is not accessible through those methods.  At last
we also note that our result is not only concerned with minimizers of
\eqref{intro-energy} but also with sequences bounded in energy. In particular,
we are able to give highly nontrivial information about the convergence of those
configurations for $\eps\rightarrow 0$.

\medskip

\textbf{Assumption of bounded support.} For both Theorem \dstwo{
  \ref{thm-compactness} } and Theorem \ref{thm-gamma}, we assume that
$\spt u_\hk{\hkk \eps} \SUS B_\rho$, uniformly for some fixed $B_\rho$. The
assumption of bounded support in Theorem \ref{thm-compactness} is used in the
proof of Lemma \ref{lem-N-low} in order to assure the boundedness of the second
moment of $u_\eps$ (in $x'-$direction) and in the proof of Theorem
\ref{thm-compactness} for the conclusion that no mass is lost in the limit
process (in direction of large $|x_1|$). For Theorem \ref{thm-compactness}, the
assumption seems to be necessary: Indeed, consider two disjoint needles of mass
$\frac{1}{2}$, with optimal cross--section and with increasing distance in
lateral direction as $\eps \to 0$. We then have the uniform bound
$E[u_\eps] \lesssim 1$, however, we do not have compactness of the sequence
$u_\eps$ in any $L^p$--space.  However, in the case of a sequence of minimizers
it might be possible to alleviate this assumption.


\section{Proofs} \label{sec-proofs}

\subsection{Preliminaries} \label{sus-prelim}

The main idea behind the compactness and the liminf
estimate in Theorem \ref{thm-gamma} is to replace the rescaled characteristic function by the cross--sectional area $A(x_1) := \A[u](x_1)$ of the domain (cf. \eqref{def-cross})  
The Fourier transform of the cross--section also appears as the first constant term of the Taylor expansion of $\widehat{u_\eps}$ in transversal direction since
\begin{align}
\widehat u(\xi_1,0) \ = \ \frac 1{(2\pi)^{\frac {n-1}2}} \widehat A(\xi_1). 
\end{align}
For the proof, we hence take into account only the low frequencies of our
configuration, i.e.\ we expand $\widehat{u}$ around $\xi'=0$ as
\begin{align} \label{u-deco} %
  \widehat{u}(\xi_1,\xi') \ %
  = \ \frac 1{(2\pi)^{\frac {n-1}2}} \widehat A(\xi) +
  \xi'\cdot\nabla_{\xi'}\widehat{u}(\xi_1,0) + \frac{1}{2}\xi'\cdot D^2_{\xi'}\widehat{u}(\xi_1,\tilde{\xi'})\xi'
\end{align}
for some $\t \xi' \in (0,\xi')$ and then perform a cut-off in transversal
direction in the Fourier space.
A corresponding lower bound for the energy is given in the next lemma:
\begin{lemma}[Lower bound on nonlocal energy] \label{lem-N-low} %
  Let $n \geq 2$ and $u\in\AA$ and let $A := \A[u]$. Let
  $\spt u \subset\R \times B_\rho'(0)$ for some $\rho>0$.  Then, for all $\eta>0$
  \begin{align} \label{cond-eta} %
    N_\eps^{(n)}[u] \ \geq \ (1- {\rho^2 \eta^2}  ) N_{\eps,\eta,1}[u] + N_{\eps,\eta,2}[u]\, ,
  \end{align}
  where
  \begin{align}
    N_{\eps,\eta,1}[u] \ &= \ \frac{\gam_n(\eps)}{(2\pi)^{n-1}} \int_{\Rn} \frac{\xi_1^2}{\eps^2\xi_1^2 + |\xi'|^2} \big|\hat A(\xi_1) \chi_{D'_{\eta}(\xi_1)}(\xi')\big|^2   \ d\xi, \label{def-N1}\\ \quad %
    N_{\eps,\eta,2}[u] \ &= \ \gam_n(\eps) \int_{\Rn} \frac{\xi_1^2}{\eps^2\xi_1^2 + |\xi'|^2} \Big|\hat u(\xi) - \frac 1{(2\pi)^{\frac{n-1}{2}}}\hat A(\xi_1)  \chi_{D'_{ \eta}(\xi_1)}(\xi')\Big|^2 \ d\xi \label{def-N2}
  \end{align}
  and where
  \begin{align}
    D'_{\eta}(\xi_1) \ %
    := \ \Big \{ \xi'\in\R^{n-1}\sd |\xi'|\leq \eta
    |\widehat{A}(\xi_1)|^{\frac 12} \Big \}.
  \end{align}
\end{lemma}
\begin{proof}
With the identity
  \begin{align}
    |\hat u|^2 \ %
    &= \ |\widehat u(\xi_1,0)|^2 \chi_{D'_{ \eta}(\xi_1)}(\xi') + |\hat u -
      \widehat u(\xi_1,0) \chi_{D'_{\eta}(\xi_1)}(\xi')|^2 \notag\\
    &\qquad  + 2 \Re \Big[(\hat u - \widehat u(\xi_1,0)) \widehat u(\xi_1,0) \chi_{D'_{ \eta}(\xi_1)}(\xi') \Big],
  \end{align}
  and since $\widehat u(\xi_1,0) = (2\pi)^{-\frac{n-1}2} \widehat A(\xi_1)$, we
  have
  \begin{align}
    N_\eps^{(n)}[u] \ %
    &\geq \ N_{\eps,\eta,1}[u] + N_{\eps,\eta,2}[u] - 2 I_\eps[u],    
  \end{align}
  where
  \begin{align*}
    I_\eps[u] \ %
    &:= \ \frac{\gam_n(\eps)}{(2\pi)^{\frac{n-1}{2}}} \Big|\int_\R \int_{D'_{ \eta}(\xi_1)} \frac{\xi_1^2}{(\eps^2\xi_1^2 + |\xi'|^2)} \Re \Big(\hat u(\xi) - (2\pi)^{-\frac{n-1}{2}}\hat A(\xi_1) \Big) \hat A(\xi_1)  \dx\xi' \dx\xi_1\dsone{\Big|}.  
  \end{align*}
  Using spherical coordinates $\xi' = (r,\ome)$ in transversal direction, we
  estimate
  \begin{align}
    I_\eps[u] \ %
      &\leq  \  \frac{\gam_n(\eps)}{(2\pi)^{\frac{n-1}{2}}}\int_\R \int_0^{\eta |\widehat{A}(\xi_1)|^{\frac 12}}   \frac{\xi_1^2 |\hat A(\xi_1)|}{\eps^2\xi_1^2 + r^2} \dsone{\Big|}\dashint_{|\xi'| = r}  \Big(\hat u(\xi) - \frac{\hat A(\xi_1)}{(2\pi)^{\frac{n-1}{2}}}\Big)   \dx\ome \Big| r^{n-2} \dx r \dx\xi_1.  \notag %
  \end{align}
  We calculate
  \begin{align} \label{theid} %
    \hat u(\xi) - \frac {\hat A(\xi_1)}{(2\pi)^{\frac{n-1}{2}}} \ %
    &= \ \frac 1{(2\pi)^{\frac n2}}\int_{\Rn} u(x) e^{-i\xi_1 x_1}
      [\cos(\xi' \cdot x') + i \sin(\xi' \cdot x') - 1 ] \dx x' \dx x_1.
  \end{align}
  The integral over the sine function vanishes by symmetry upon integration
    in $\xi'$ over $|\xi'| = r$. Expanding the cosine function and since
  $\spt u \subset \R\times B_\rho'(0)$ we hence get
  \begin{align} \label{to-ins} %
    \Big|\dashint_{|\xi'| = r} \Big(\hat u(\xi) - \frac{\hat
    A(\xi_1)}{(2\pi)^{\frac{n-1}{2}}}\Big) \dx\ome \Big| \ %
    &\lupref{theid}\leq \ \frac {r^2}{2(2\pi)^{\frac n2}} \int_{\R \times B_\rho'(0)} |x'|^2 |u(x)| \dx x \ %
      \leq \ \frac {r^2 \rho^2}{2(2\pi)^{\frac n2}}.
  \end{align}
  Inserting \eqref{to-ins} into \hk{the expression for $I$} and since
  $r \leq \eta|\widehat{A}(\xi_1)|^{\frac 12}$ we obtain
  \begin{align}
    2I_\eps[u] \ %
    &\leq  \  \frac{2 \gam_n(\eps)}{(2\pi)^{n - \frac 12}}\int_\R \int_0^{\eta |\widehat{A}(\xi_1)|^{\frac 12}}  \frac{\xi_1^2}{(\eps^2\xi_1^2 + r^2)} |\hat A(\xi_1)| \Big| \frac {r^2 \rho^2}{2} \ r^{n-2} \dx r \dx\xi_1  \notag %
    \leq \ \frac {\eta^2 \rho^2 }{\sqrt{2\pi}}  N_{\eps,\eta,1}[u_\eps].
  \end{align}
\end{proof}

Since $\widehat A$ does not depend on $\xi'$, in order to
  calculate the integral for $N_{\eps,\eta,1}[u]$, it is useful to first integrate in
  $\xi'$ and then in $\xi_1$. For this, the following integral formula is  useful:
  
\begin{lemma}[Integral formula] \label{lem-intform} %
  Let $\beta\geq 0$, $k\in\mathbb{N}$ and let $\xi' \in \R^{n-1}$. Then
  \begin{align} 
    \int_{|\xi'|\leq \beta} &\frac{\xi_1^2 |\xi'|^{2k}}{\eps^2\xi_1^2 + |\xi'|^2} \dx\xi' \ %
                              = \ \frac{\ome_{n-2}}{n-1+2k} \ \eps^{-2} 
        \beta^{n-1+2k} 
        \f\left(1,\tfrac{n-1+2k}{2},\tfrac{n+1+2k}{2},-\left(\tfrac{\beta}{\eps|\xi_1|}\right)^2\right)\, , \nonumber
\end{align}
where $\f$ is the hypergeometric function.
In particular for $k=0$
\begin{align} \label{calc_inner_int} %
\int_{|\xi'|\leq \beta}
\frac{\xi_1^2}{\eps^2\xi_1^2 + |\xi'|^2} \dx\xi' %
\ = \
\begin{TC}
\frac{2}{\eps} |\xi_1|
\arctan\Big(\frac{\beta}{\eps
|\xi_1|}\Big) &\text{for $n=2$,} \vspace{0.6ex}\\ %
\pi\xi_1^2 \ln\Big(1+ \Big( \frac{\beta}{\eps|\xi_1|} \Big)^2 \Big) &\text{for $n=3$}.
\end{TC}
\end{align}
Furthermore, if $n+2k\geq 4$ and $\frac{\beta}{\eps|\xi_1|}\rightarrow\infty$ we have the asymptotic expansion
\begin{align} \label{calc_inner_int_expansion}
\int_{|\xi'|\leq \beta} \frac{\xi_1^2}{\eps^2\xi_1^2 + |\xi'|^2} |\xi'|^{2k} \dx \xi'
\ &= \ \frac{\ome_{n-2}}{n-3+2k} \xi_1^2 \beta^{n-3+2k} + O\left(\eps\xi_1^3\beta^{n-4+2k}\right)\, .
\end{align}
\end{lemma}

\begin{proof}
  We write $N :=n-2+2k$. With the the change of variables
 $s = \frac{|\xi'|}{\eps|\xi_1|}$, we get
\begin{align*}
\int_{|\xi'|\leq \beta} \frac{\xi_1^2}{\eps^2\xi_1^2 + |\xi'|^2} |\xi'|^{2k} \dx \xi' 
\ &= \ \ome_{n-2} \eps^{N-1} \xi_1^{N+1} \ \int_{0}^{\beta/(\eps|\xi_1|)} \frac{s^{N}}{1 + s^2}  \dx s\, .
\end{align*}
The last integral can be expressed as a hypergeometric function \cite[15.6.1]{DLMF2021} 
\begin{align*}
  \int_0^{\frac{\beta}{\eps|\xi_1|}} \frac{s^N}{1 + s^2}  \dx s 
  \ = \ \Big(\frac{\beta}{\eps|\xi_1|}\Big)^{N+1} \frac 1{N+1} \ \f\Big(1,\frac{N+1}{2},\frac{N+3}{2},-\left(\frac{\beta}{\eps|\xi_1|}\right)^2\Big)\, ,
\end{align*}
which yields \hk{the assertion}. Then \eqref{calc_inner_int} directly follows
by \cite[Sec. 7.3.2 (148),p.476]{Prudnikov1990} for $n=3$ and by
\cite[Sec. 7.3.2 (83),p.473]{Prudnikov1990} for $n=2$ since
$\f(1,\frac12,\frac32,z)=\f(\frac12,1,\frac32,z)$ (see \cite[Sec. 7.3.1 (4),
p.454]{Prudnikov1990}).

Now assume that $N\geq 2$. With the help of \cite[Sec. 7.3.1 (6)]{Prudnikov1990}
we also get the expansion
\begin{align*}
\f\left(1,\tfrac{N+1}{2},\tfrac{N+3}{2},-s^2\right) 
\ &= \ \frac{(N+1)\Gam(\hk{\frac 12 (N-1)})}{2\Gam(\hk{\frac 12 (N+1)})} \frac{1}{s^2} + O\Big(\frac{1}{s^3}\Big)\, .
\end{align*}
Together with $\Gam((N-1)/2)/\Gam((N+1)/2)=\frac{2}{N-1}$
\eqref{calc_inner_int_expansion} then follows.
\end{proof}

\subsection{Proof of Theorem \ref{thm-compactness}}
\label{subsec-cptness}

We consider a sequence of functions $u_\eps \in \AA$ with
$\supp u_\eps \subset B_\rho(0)$ (for some fixed $\rho>0$) satisfying the
uniform bound $\dstwo{E^{(n)}_\eps}[u_\eps] \leq C$. The compactness result is
based on the decomposition in Lemma \ref{lem-N-low}. This decomposition gives a
lower bound on the nonlocal energy in terms of the cross--section \dstwo{area}
$A_\eps := \A[u_\eps]$, i.e.
  \begin{align}
    N_\eps^{(n)}[u_\eps] \ \geq \ (1- \rho^2 \eta^2) \frac{\gam_n(\eps)}{(2\pi)^{n-1}} \int_{|\xi'|\leq \eta |\widehat{A}(\xi_1)|^{\frac 12}} \frac{\xi_1^2}{\eps^2\xi_1^2 + |\xi'|^2} \big|\hat A_\eps(\xi_1) \chi_{D'_{\eta}(\xi_1)}(\xi')\big|^2   \dx\xi
  \end{align}
  \hk{noting that the prefactor is positive for $\eta$ sufficiently small.} At
  the core of the proof is the next lemma which allows us to get uniform bounds
  on Sobolev norms for $A_\eps$:
  \begin{lemma}[Bound on cross--section $A$] \label{lem-com1} %
    Let $n=2,3$. Let $u\in \AA$ with $\spt u \subset \R\times B_\rho'(0)$ for
    some $\rho>2$ and let $A := \A[u]$. Then we have the bounds
  \begin{enumerate}
  \item\label{it-bound-hs}  
    For any $s \in [0,\frac{1}{14}]$ if $n = 2$ and $s \in [0,\frac 1{10})$ if $n = 3$, we have
    \begin{align} \label{eq-bound-hs} %
      \int_\R |\xi_1|^{2s} |\widehat A(\xi_1)|^2 \dx\xi_1 \ %
      \lesssim \ P_\eps^{(n)}[u]^2 + N_\eps^{(n)}[u] + \hk{C_{\rho,s}} ,
    \end{align}
  \item\label{it-bound-h12} %
    \hktwo{Let $\eta \in (0,\frac{1}{2\rho})$}.  Then we have
    \begin{align} \label{eq-bound-Seps-hs} %
      \int_{{S_{\eps,\eta}}} |\xi_1|^{n-1} |\widehat{{A}}(\xi_1)|^2 \ d\xi_1 \
    \lesssim \ N_\eps^{(n)}[u] \quad 
    \end{align}
  for
  \begin{align} \label{def-Seps} %
    {S_{\eps,\eta}} \ := \ \Big\{ \xi_1\in\R \ : \ \hk{q_{\eps,\eta}(\xi_1)}
    \geq \eps^{\frac{\eta}{2}-1} \Big\} \qquad %
    \text{where
    ${q_{\eps,\eta}}(\xi_1) \ := \ \frac{\eta|\widehat{{A}}(\xi_1)|^{\frac
    12}}{\eps|\xi_1|}$.}
  \end{align}
  \end{enumerate}
\end{lemma}
\begin{proof}
\hk{Let $\eta \in (0,\frac{1}{2\rho})$ to be fixed later}. By\hk{Lemma
    \ref{lem-N-low} and \ref{lem-intform} we have}
  \begin{align} %
    N_\eps^{(n)}[u] \ %
    &\gtrsim \ \hk{\gam_n(\eps)} \int_{\R} \bigg(\int_{|\xi'|\leq
      \eta|\widehat{{A}}(\xi_1)|^\frac12} \frac{\xi_1^2}{\eps^2 \xi_1^2 +
      |\xi'|^2}  \dx\xi' \bigg)  |\widehat{{A}}(\xi_1)|^2 \dx\xi_1 \\ 
    &\gtrsim \ \int_{\R}
      \left [
      \begin{array}{l}
        |\xi_1| \arctan(q_{\eps,\eta}(\xi_1)) \vspace{0.6ex}\\ %
        \frac 1{|\ln \eps|} \xi_1^2 \ln (1+ {q_{\eps,\eta}}(\xi_1)^2)  
      \end{array}
    \right ]
    |\widehat{{A}}(\xi_1)|^2 \dx\xi_1
    \qquad %
    \text{for }
    \begin{TC}
      n=2, \vspace{0.6ex}\\ %
      n=3.
    \end{TC} \label{comp-1}
  \end{align}
  With the control of the anisotropic perimeter, we also have 
  \begin{align}
    P_\eps^{(n)}[u] \ %
    \hk{\geq} \ \eps \int_{\Rn} |\partial_1 u(x)| \dx x_1 \ %
    \gtrsim \ \eps |\widehat{\partial_1 {A}}(\xi_1)| 
      \hk{=} \ \eps |\xi_1| |\widehat{{A}}(\xi_1)| \ %
      \qquad \FA{\xi_1 \in \R}. \ \label{con-2} %
\end{align}
The proofs for (i)--(ii) are based on the lower bounds in
  \eqref{comp-1}--\eqref{con-2} and are given below:

  \medskip
  
  \textit{\ref{it-bound-hs}:} We
  use the decomposition
  $\R = (R_{\eps,\eta}^{(n)}) \cup (R_{\eps,\eta}^{(n)})^c$ in frequency space
  where
  \begin{align} \label{def-set-R_eps-cpt-3} %
    \dstwo{R_{\eps,\eta}^{(n)}} \ %
    := \ \Big\{ \xi_1\in\R \ : \ {q_{\eps,\eta}}(\xi_1)
    \geq \left[
    \begin{array}{ll}
      1, \qquad &\text{for $n = 2$} \\
      \eps^{-\eta}, &\text{for $n = 3$}     
    \end{array}
                      \right]
                      \Big\}.
  \end{align}
  \hkk{using the notation \eqref{def-Seps}.} The factor $\eps^{-\eta}$ for
  $n = 3$ \hkk{ensures} that the logarithm \hkk{in} \eqref{comp-1} can be
  \hkk{controlled}. Indeed, for $\xi_1 \in \dstwo{R_{\eps,\eta}^{(n)}}$ we have
  $\arctan(q_{\eps,\eta}(\xi_1)) \gtrsim 1$ and
  $\ln (1+ {q_{\eps,\eta}}(\xi_1)^2) \gtrsim |\ln(\eps)|$. Hence, restricting
  integration to $R_{\eps,\eta} := R_{\eps,\eta}^{(n)}$, from \eqref{comp-1} we
  get
  \begin{align} %
    \int_{R_{\eps,\eta}}
    |\xi_1|^{n-1}
    |\widehat {A}(\xi_1)|^2 \dx\xi_1 \ %
    \lesssim \ \dstwo{N^{(n)}_\eps}[{u}]. %
    \label{comp-2}
  \end{align}
  In particular for $s<\frac12(n-1)$ we get
  \begin{align} %
    \int_{R_{\eps,\eta}} |\xi_1|^{2s} |\widehat{{{A}}}(\xi_1)|^2 \dx\xi_1  \notag
    \ &\leq \ 2  \NI{\widehat A} + \dstwo{\Big(}\sup_{\xi_1\in \dstwo{R_{\eps,\eta}}, |\xi_1| \geq 1} \dstwo{|\xi_1|^{2s-(n-1)}} \dstwo{\Big)}\int_{\dstwo{R_{\eps,\eta}}} |\xi_1|^{n-1}
        |\widehat{A}(\xi_1)|^2 \dx\xi_1 \notag \\
    \ &\lesssim \ N_\eps^{(n)}[u] + \hk{1}. \label{cpt_unif_bound_Hs}
  \end{align}
  The above calculations show that \ref{it-bound-hs} holds with integration
  restricted to $R_{\eps,\eta}$.  To show a corresponding estimate on
  $R_{\eps,\eta}^c$ we use the further decomposition
  \begin{align*}
    R_\eps^c \ = \ \bigcup_{k=0}^{N} I_k, \qquad %
    \text{where }I_k \ = \ \Big \{ \xi_1 \in R_\eps^c \ : \ \big(\frac 1\eps \big)^{\alp_k} \leq |\xi_1| \leq \big(\frac 1\eps \big)^{\alp_{k+1}} \Big \},
  \end{align*}
  $I_0:=\dstwo{\{\xi_1\in R_{\eps,\eta}^c\sd 0\leq |\xi_1|\leq \eps^{-\alpha_1}\} } $ and $I_N:=\dstwo{\{\xi_1\in R_{\eps,\eta}^c\sd \eps^{-\alpha_N}\leq |\xi_1|\leq \infty\} }$ 
  and for a finite set of numbers $0 < \alp_1 < \ldots < \alp_{N} < \infty$ with
  $N \in \N$ to be chosen in the sequel.

  \medskip

  With the choice $\alp_N := \frac{2}{1-2s}$ and by \eqref{con-2} we get
  \begin{align} \label{est-alpha_infty} %
    \int_{I_N} |\xi_1|^{2s} |\widehat{A}(\xi_1)|^2 \dx\xi_1 %
    \  &\lupref{con-2}\lesssim \ \frac{(P_\eps^{(n)}[u])^2}{\eps^{2}} \int_{\eps^{-\alp_N}}^\infty |\xi_1|^{2s-2} \dx\xi_1 \\
    \ &\lesssim \ \eps^{(1-2s)\alp_N - 2} (P_\eps^{(n)}[u])^2 \ %
        \ \lesssim \ (P_\eps^{(n)}[u])^2 \ . \nonumber
  \end{align}
  It remains to show the corresponding estimates on $I_k$ for
  $k \leq N-1$. We give the proof separately for $n = 2,3$ with a
  different choice of numbers $\alp_k$ for $2 \leq k \leq N-1$.

\medskip

\textit{The case $n = 2$:} In this case we decompose
$R_{\eps,\eta}^c = I_0 \cup I_1$, i.e. $N = 1$ and the estimate for $I_1 = I_N$
has been given above. By construction, we have
$\NIL{q_{\eps,\eta}}{R_{\eps,\eta}^c} \leq 1$. Since $\arctan(q) \gtrsim q$ for
$q \leq 1$ from \eqref{comp-1} and by definition of $q_{\eps,\eta}$ we hence
immediately get the estimate
\begin{align} \label{A52-est} %
  \int_{I_0} |\widehat A(\xi_1)|^{\frac 52} \ d\xi_1 \ \leq \ C_\eta \:  \eps \: N_\eps^{(n)}[u].
\end{align}
By Hölder's inequality, it follows that
\begin{align*}
  \int_{I_0} |\xi_1|^{2s}|\hat{A}(\xi_1)|^2\dx\xi_1
  \ &\leq \ \NIL{\xi_1}{I_0}^{2s}   \NPL{|\hat{A}|^{\frac 52}}{1}{I_0}^{\frac 45} |I_0|^{\frac 15} \ %
  \ \upref{A52-est}\leq \ C_\eta \eps^{- 2s \alp_{1}   - \frac 15 \alp_1  + \frac 45} (N_\eps^{(n)}[u])^{\frac 45}.
\end{align*}
The right hand side above is uniformly bounded in $\eps$ if $\alp_1 := \frac {4}{1 + 10s}$ 
\hk{and by Young's inequality for $\eta:=\frac{1}{4\rho}$ we get $C_\eta (N_\eps^{(n)}[u])^{\frac 45} \leq C_\rho + C N_\eps^{(n)}[u]$.}
To close the estima\dsk{t}e we need
\dstwo{$\alp_1 \geq \alp_N$}. That is we need
$\frac {4}{1 + 10s} \geq \frac{2}{1-2s}$ or \dsk{equivalently} $s \leq \frac 1{14}$.

\medskip

\textit{The case $n = 3$:}
  For $\xi_1 \in R_{\eps,\eta}^c$ we have
  $|\hat{A}|\leq \eta^{-2} \eps^{2(1-\eta)}|\xi_1|^2$. With
  $\alp_1:=\frac{4(1-\eta)}{5+2s}$ we hence get
\begin{align} \label{proof:them:cpt:3d:decom_0} %
  \int_{I_0} |\xi_1|^{2s}|\hat{A}(\xi_1)|^2\dx\xi_1 \ \lesssim \
  \frac{\eps^{4(1-\eta)}}{\eta^4} %
  \hk{\int_0^{\eps^{-\alp_1}} \xi_1^{4+2s} \ d\xi_1} \ %
  \leq \ C_\eta \eps^{4(1-\eta)-(5+2s)\alp_1} \ %
  \leq \ \hk{C_\eta}\, .
\end{align} 
It remains the estimates on $I_k$ for $k=1,...,N-1$.  Noting that
$|I_k| \leq \eps^{-\alp_{k+1}}$ and
$\hk{\eps^{-\alp_{k}}} \leq |\xi_1| \leq \eps^{-\alp_{k+1}}$ in $I_k$,
\hk{estimate} \eqref{comp-1} implies
\begin{align} \label{proof:them:cpt:3d:decom_1} N_\eps^{(3)}[u] \ %
  &\gtrsim \ \frac 1{|\ln \eps|} \int_{I_k} \eps^{-2\alp_k} \ln\left(1+\eta^2 \eps^{2\alp_{k+1} - 2} |\widehat{A}(\xi_1)|\right) |\widehat{A}(\xi_1)|^2\ d\xi_1 \\
  &\geq \frac{C_\eta}{|\ln \eps|} \eps^{2\alp_{k+1} - 2-2\alp_k} \int_{I_k}
    |\widehat{A}(\xi_1)|^3\ d\xi_1, \notag
\end{align}
\hk{using that $\ln(1+y) \gtrsim y$ for $y \leq 1$}.  This holds if
$\eps^{2\alp_{k+1}-2}|\widehat{A}|\leq \eps^{2\alp_{k+1}-2} \hk{\leq 1}$
for $\hk{\eps \leq 1}$ if $\alp_{2} > 1$ and hence $\alp_{k+1} > 1$ for
$k \geq 1$ \hk{(this condition holds, see \eqref{def-akp1} below)}.  Hence,
\begin{align*}
  \int_{I_k} |\widehat A(\xi_1)|^3 \ d\xi_1
  \ \leq \ C_\eta |\ln \eps| \eps^{2(1+\alp_k)- 2 \alp_{k+1}} 
  N_\eps^{(3)}[u]
\end{align*}
For $k=1,...,N-1$ we hence get by Hölder's inequality
\begin{align}
\int_{I_k} |\xi_1|^{2s}|\hat{A}(\xi_1)|^2\dx\xi_1
\ &\lesssim \ \NIL{\xi_1}{I_k}^{2s}  |I_k|^{\frac 13} \Vert |\widehat A|^3\Vert_{L^1(I_k)}^{\frac23} \nonumber \\
\ &\leq \ C_\eta |\ln \eps|^{\frac 23} \eps^{\frac 43(1+\alpha_k) - (\frac 53 + 2s)\alp_{k+1}} (N_\eps^{(3)}[u])^\frac23.  \label{Hs-N-3D}
\end{align}
\hk{As in two dimensions we then apply Young's inequality to \eqref{Hs-N-3D} and
  note that} in order to be bounded, the exponent of $\eps$ has to be strictly
positive.  This motivates to define
\begin{align} \label{def-akp1} %
\alp_{k+1} \ := \ \frac{4(\alp_k+1-\eta)}{5+6s}.
\end{align}
We note that for $\hktwo{\eta < \frac{15}{4}(\frac{1}{10}-s) }$, this definition
produces a monotonically increasing sequence of $\alpha_k$ such that
$\alpha_2>1$ and $\alp_{k+1} - \alp_k \geq \frac{15}{5+6s}(\frac{1}{10} - s)  - \frac{4}{5+6s}\eta > 0$. In particular, since $\alp_2 > 1$ the Taylor approximation in \eqref{proof:them:cpt:3d:decom_1} is justified.  
Furthermore, for $0<s<\frac{1}{10}$ after finitely
many iterations we obtain $\alpha_{k+1}>\alpha_N$ which terminates the
algorithm.

\DETAILS{
$$ 
\alpha_{k+1}-\alpha_k 
\geq \frac{1}{5+6s}(\alpha_k(-1-6s) +4(1-\eta))
\geq \frac{15}{5+6s}(\frac{1}{10} - s)  - \frac{4}{5+6s}\eta
>0.
$$ 
for 
$$
\eta < \frac{15(1/10 - s)}{4} = \frac{(3/2 - 15 s)}{4} 
$$
}

\medskip

  \textit{\ref{it-bound-h12}:} For $\xi_1 \in S_{\eps,\eta}$ \hk{and since
    $\eta < \frac 14$}, we have
  $q_{\eps,\eta}(\xi_1) \hk{\geq C_\eta \eps^{\frac{\eta}{2}-1} \geq}
  C_\eta$. This yields $\arctan(q_{\eps,\eta}(\xi_1)) \hk{\geq C_\eta}$ and
    $\ln (1+q_{\eps,\eta}(\xi_1)^2) \hk{\geq C_\eta |\ln \eps|}$.  Inserting
    these estimates into \eqref{comp-1} yields assertion \ref{it-bound-h12}.
\end{proof}

We next give the proof of Theorem \ref{thm-compactness} for
  $n = 2,3$. Since the estimates have already been carried out in Lemma \ref{lem-com1}, it remains to give the functional analytic argument based on weak compactness in $H^s$ together with an additional argument \dsk{showing} that the weak limit is in the (better) space $H^\frac{n-1}{2}$.
\dsone{The proof for $n=4$ proceeds differently since the estimates  on the set $R_{\eps,\eta}$ coming from the hypergeometric function obtained in Lemma \ref{lem-intform} do no longer directly provide us with a control over a $H^s-$norm of $A$. Instead, an additional step must be carried out which we will detail in the appendix.}

\begin{proposition}[Compactness for $n \in \{ 2, 3 \}$] \label{prp-com23} %
  Let $n = 2,3$. For any sequence $u_\eps\in \AA$ with
  $E_\eps^{(n)}[u_\eps] \leq C$ and $\spt u_\eps \subset B_\rho(0)$
  for some $\rho>0$ (up to a subsequence) we have
  \begin{align} \label{top-conv2} %
    \A[u_\eps] \ \wto \ A \quad\text{ in } H^s(\R) \quad \text{as $\eps \to 0$},
  \end{align}
  for any fixed $s \in [0,\frac{1}{14}]$ if $n = 2$ and
  $s \in [0,\frac 1{10})$ if $n = 3$ and for some $A \ \in \ \AA_0^{(n)}$.
\end{proposition}
\begin{proof} %
  Assume that $s$ satisfies the assumptions above.  Since $u_\eps$ satisfies
  $\Vert u_\eps\Vert_{L^1(\R^n)}\leq 1$ and
  $\Vert u_\eps\Vert_{L^\infty(\R^n)}\leq 1$, \dsone{we also have
    $\NT{u_\eps} \leq 1$ and} by the Banach-Alaoglu theorem, for a subsequence
  (still denoted by $u_\eps$) we have $u_\eps\wto u$ in $L^2$ and
  $u_\eps\stackrel{*}{\wto} u$ in $L^\infty$ for some $u \in L^\infty$. We write
  $A_\eps:=\A[u_\eps]$ and $A:=\A[u]$.  By Fubini's theorem we then also get
  $A_\eps\wto A$ in $L^2$ and $A_\eps\stackrel{*}{\wto} A$ in $L^\infty$. From
  Lemma \ref{lem-com1}(i) it follows that
  $\NNN{A_\eps}{\dot H^s} \leq \dstwo{C_{\rho,s}}$ uniformly \dstwo{in
    $\hk{\epsilon \in (0,1)}$}.  After selection of a subsequence and by
  Banach-Alaoglu we then have $A_\eps \wto A$ in $H^s$.

  \medskip
  
  Since $\spt u_\eps \SUS B_\rho$ we have $\spt u \SUS B_\rho$ and
  $\int_{\R^n} u(x) \dx x=1$.  It remains to show that
  $A \in H^{\frac{n-1}{2}}(\R)$: \dstwo{We fix $\eta := \frac 1{4\rho}$.}  By
  Lemma \ref{lem-com1} \ref{it-bound-h12} \dsone{and since
    $\NT{A_\eps} \leq C_\rho$, we have the uniform estimate
    $\NNN{(\widehat{A_\eps}\chi_{S_{\eps,\eta}})^\vee}{H^\frac{n-1}{2}(\R)}\leq
    C_{\rho}$. Since $H^{\frac{n-1}2}(\R)$ is a reflexive space and by
    application of Banach-Alaoglu there exists $B\in H^{\frac{n-1}2}(\R)$ such
    that, after selection of a subsequence }
    \begin{align} \label{Bcon} %
          (\widehat{A_\eps}\chi_{S_{\eps,\eta}})^\vee \ \wto \ B \qquad \dsone{\text{ in } H^{\frac{n-1}2}(\R)} \hkone{\text{ as $\eps \to 0$.}}
    \end{align}
    The global energy bound, together with the definition of $S_{\eps,\eta}$ in
    \eqref{def-Seps} yields, after extracting a subsequence, that
    $|\widehat{A_\eps}|\chi_{\R\setminus S_{\eps,\eta}} \to 0$ a.e. \hk{as
      $\eps \to 0$}.  Therefore $\chi_{\R\setminus S_{\eps,\eta}}\to 0$
    a.e. \hk{as $\eps \to 0$} by the Paley–Wiener theorem.  Since
    $\NI{\chi_{S_{\eps,\eta}}} \leq 1$ and $A_\eps \wto A$, by Lemma
    \ref{lem-weak_pointwise_conv} we then have
    $\widehat{A_\eps}\chi_{S_{\eps,\eta}} \wto \widehat{A}$ in $L^2$ \hk{as
      $\eps \to 0$}. By \eqref{Bcon} and the uniqueness of the weak limit this
    implies $A=B \in H^{\frac{n-1}{2}}$.
\end{proof}
We have used the following auxiliary result in the proof of Proposition
\ref{prp-com23}:
\begin{lemma}[Auxiliary lemma] \label{lem-weak_pointwise_conv} %
 Let $g_n\in L^2(\R),\dstwo{f_n\in L^\infty(\mathbb{R})}$\hk{, $n \in \N$,} be two sequences with $\NI{f_n} \lesssim 1$,
    $f_n\to f$ pointwise a.e. and $g_n\wto g$ in $L^2$ for some
    $g\in L^2(\R),\dstwo{f\in L^\infty(\mathbb{R})}$. Then
  \begin{align}
    f_n\cdot g_n \ \wto \ f\cdot g \qquad \text{ in  $L^2(\R)$}\hk{\text{ as $n \to \infty$}}.
  \end{align}
\end{lemma}
\begin{proof}
 For $\varphi\in L^2(\R)$, we write
  \begin{align}
    |\skp{\varphi}{f_n g_n} - \skp{\varphi}{fg}| \ %
    &\leq \ \NT{g_n} \NT{\varphi (f_n - f)} + |\skp{\varphi f}{(g_n - g)}| \ %
      =: \  I_1 + I_2.  
  \end{align}
Since $|\varphi (f_n - f)| \ \lesssim \ |\varphi| \in L^2$ we get
    $I_1\to 0$ for $n\to \infty$ by dominated convergence.  Furthermore, since
  $\varphi f\in L^2(\R)$ and by the weak convergence
  of $g_n$ we have $I_2\to 0$.
\end{proof}

\subsection{Proof of Theorem \ref{thm-gamma}}

The proof of Theorem \ref{thm-gamma} follows from the liminf estimate in
Proposition \ref{prp-lower} together with the construction of the recovery
sequence in Proposition \ref{prp-upper}.  As in the proof of Theorem
\ref{thm-compactness}, the liminf is based on the decomposition of the nonlinear
energy from Lemma \ref{lem-N-low}. To get the optimal leading order constant for
our estimate, we additionally take the limit $\eta \to 0$.  For the upper bound
we use a constant recovery sequence in our rescaled setting (for \dsone{sufficiently} smooth limit function $A$).

\begin{proposition}[Liminf estimate] \label{prp-lower} Let $n \in \{ 2, 3 \}$. Then for
  any sequence $u_\eps\in \AA$ such that \eqref{top-conv} holds, we have
  $\liminf_{\eps\to 0} E_\eps^{(n)}[u_\eps] \ \geq \ E_0^{(n)}[A]$.
\end{proposition}
\begin{proof}
  To estimate the nonlocal part for both $n=2,3$, we will use
  \eqref{calc_inner_int} and restrict the integration in $\xi_1$--direction to
  the set ${S_{\eps,\eta}}$ defined in \eqref{def-Seps}.
    
  \medskip
  
  \textit{The case $n=3$:} Taking a subsequence, we can assume $A_\eps \to A$
  a.e.. By the assumption $\spt u_\eps \SUS B_\rho$ we also have
  $\spt A_\eps \SUS [-\rho,\rho]$. In view of Theorem \ref{thm-compactness} we
  then get $\NP{\sqrt{A_\eps}}{1} \to \NP{\sqrt{A}}{1}$.  By the isoperimetric
  inequality, this implies
  \begin{align} \label{calc_per_3dim} %
    P_\eps^{(n)}[u_\eps] \ &\geq \ \int_{\R^3}
    |\nabla'u_\eps(x)| \dx x \ \geq \ 2\sqrt{\pi} \int_\R \sqrt{A_\eps(x_1)} \
    \dx x_1 \to \ 2\sqrt{\pi} \int_\R \sqrt{A(x_1)} \dx x_1.
  \end{align}
  For the estimate of the nonlocal term we assume that $\eta \in (0,\frac{1}{4})$
  is small enough such that $1 - {\rho^2}\eta^2 > 0$ in \eqref{cond-eta}. Using
  \eqref{calc_inner_int} and \dstwo{since by Lemma \ref{lem-com1} (ii) it holds
    that} $(\widehat A_\eps \chi_{{S_{\eps,\eta}}})^\vee\wto A$ in $H^1$ \hk{for
    $\eps \to 0$ for fixed $\eta$} we get
  \begin{align} \label{calc_N_3dim_liminf}
    \liminf_{\eps\to 0} N_\eps^{(n)}[u_\eps] \ &\lupref{cond-eta}{\geq} \
                                                 \liminf_{\eps\rightarrow 0} \frac{1}{(2\pi)^2} \frac{1-{\rho^2}\eta^2}{7|\ln\eps|} \int_\R \int_{|\xi'|\leq \eta |\widehat{A_\eps}(\xi_1)|^{\frac 12}} \frac{\xi_1^2}{\eps^2\xi_1^2 + |\xi'|^2} \big|\widehat{ A_\eps}(\xi_1)\big|^2   \dx\xi' \dx\xi_1 \nt \\
                                               &\lupref{calc_inner_int}\geq \ \frac{1-{\rho^2}\eta^2}{28\pi}\  \liminf_{\eps\to  0} \frac{\ln(1+\eps^{\eta-2})}{|\ln\eps|} \int_{{S_{\eps,\eta}}} \xi_1^2 |\widehat{A_\eps}(\xi_1)|^2 \dx\xi_1 \nt \\
     &\geq \frac{1}{28\pi}(2-\eta)(1-{\rho^2}\eta^2) \liminf_{\eps\to  0} \int_{{S_{\eps,\eta}}} \xi_1^2|\widehat{ A_\eps}(\xi_1)|^2 \dx\xi_1  \nonumber \\
                                               &\geq \ \frac{1}{28\pi}(2-C{_\rho}\eta) \int_{\R} \xi_1^2|\widehat{A}(\xi_1)|^2 \dx\xi_1,
  \end{align}
  by the weak lower semicontinuity of the $\dot{H}^1-$norm.  Combining \eqref{calc_per_3dim} with \eqref{calc_N_3dim_liminf} yields the lower bound
  for $E_\eps^{(3)}[u_\eps]$ for $\eta \to 0$.
  
\medskip
  
\textit{The case $n=2$:} The perimeter of any \dsone{measurable subset of $\R$ with positive measure} is estimated from below by $2$. \dsone{In particular, if $A_\eps(x_1)>0$ this implies that $\int_\R |\nabla'u_\eps(x_1,x')|\dx x'\geq 2$.}
By the lower semicontinuity of the \ds{measure of the positivity set}, this implies
  \begin{align} \label{calc_per_2dim} %
    \liminf_{\eps\to 0} P_\eps^{(n)}[u_\eps] \ \lupref{def-Peps}\geq \
    \liminf_{\eps\to 0} \ \int_{\R^2} |\nabla' u_\eps(x)| \dx x \ \geq \ 2 \ds{|\{x\in\R\sd A(x)>0\}|}.
  \end{align}
  As before we assume that $\eta \in (0,\frac 14)$ is
  small enough such that $1 - {\rho^2}\eta^2 > 0$ in \eqref{cond-eta}. With
  \eqref{calc_inner_int} and the definition of ${S_{\eps,\eta}}$ in
  \eqref{def-Seps} we get
  \begin{align} \label{calc_N_2dim_liminf}
      \liminf_{\eps\to 0} N_\eps^{(n)}[u_\eps] \ &\lupref{cond-eta}{\geq} \
       \liminf_{\eps\rightarrow 0} \frac{1}{2\pi} (1-{\rho^2}\eta^2) \eps \int_\R \int_{|\xi'|\leq \eta
          |\widehat{A_\eps}(\xi_1)|^{\frac 12}} \frac{\xi_1^2}{\eps^2\xi_1^2 + |\xi'|^2} \big|\widehat{ A_\eps}(\xi_1)\big|^2   \dx\xi' \dx\xi_1 \nt \\
      &\lupref{calc_inner_int}\geq \ \frac 1\pi (1-{\rho^2}\eta^2)\  \liminf_{\eps\to  0} \Big( \arctan (\eps^{\frac{\eta}2 -1}) \int_{{S_{\eps,\eta}}} |\xi_1| |\widehat{A_\eps}(\xi_1)|^2 \dx\xi_1 \Big) \nt \\
      &\geq \ \frac12(1-{\rho^2}\eta^2) \int_\R |\xi_1| |\widehat{A}(\xi_1)|^2 \dx\xi_1,
  \end{align}
  since $(\widehat A_\eps \chi_{{S_{\eps,\eta}}})^\vee\wto A$ in $H^{\frac 12}$ (\dstwo{by Lemma \ref{lem-com1} (ii)}), by weak lower semicontinuity of the $\hhom$--norm and since
  $\arctan (\eps^{\frac{\eta}2 -1}) \to \frac \pi2$ for $\eps \to 0$. The result
  follows by adding \eqref{calc_per_2dim} and \eqref{calc_N_2dim_liminf} and
  taking the limit $\eta \to 0$.
\end{proof}

Before we give the proof of the recovery sequence, we derive a series
representation for the Fourier transform of a symmetric characteristic
function. In our application, we will consider functions of the form
$A(x_1) = \ome_{n-1} \hkk{\rho}^{n-1}(x_1)$:
\begin{lemma}[Fourier transform for rotationally symmetric
  set] \label{lem-app-ft} %
  Let $n \in \N$.  For $u \in BV(\Rn)$ of the form
  $u(x_1,x') := \chi_{\left(0,{\hkk{\rho}}(x_1)\right)}(|x'|)$ for
  ${\hkk{\rho}} \in \cciL{\R,[0,\infty)}$, we have
  \begin{align} \label{app-prop-ft-n} %
    \widehat u(\xi) \ %
    &= \ \frac{\ome_{n-1}}{(2\pi)^{\frac n2}}\sum_{k=0}^\infty \frac{(-1)^k
      |\xi'|^{2k}}{k!\Gam(k+1+\frac{n-1}{2})2^{2k+\frac{n-1}{2}}} \int_{\spt
      {\hkk{\rho}}} e^{-i x_1\xi_1} {\hkk{\rho}}(x_1)^{2k + n-1} \dx x_1
  \end{align}
  and there exists a constant $C<\infty$ such that
  \begin{align} \label{app-prop-ft-n-decay} |\widehat{u}(\xi_1,\xi')| \ \leq \ C
    |\xi'|^{-\frac{n}{2}} \int_{\spt \hkk{\rho}} \hkk{\rho}^{\frac n2 - 1}(x_1)\dx x_1 .
  \end{align}
\end{lemma}
\begin{proof}
  We apply the Hankel transform to the radially symmetric function
  $u(\cdot,\xi')$ for fixed $x_1\in\R$ (cf. \cite[B.5]{Grafakos2009}). This
  yields
  \begin{equation} \label{app-prop-ft-dim-n}
    \begin{split}
      \widehat{{u}}(\xi_1,\xi') \ %
      &= \ \frac{\ome_{n-1}}{(2\pi)^{\frac n2}} \int_{\spt {\hkk{\rho}}}
      e^{-i x_1\xi_1} \frac{1}{|\xi'|^{\frac{n-3}{2}}} \int_{0}^\infty {\sig}^{\frac{n-1}{2}} \chi_{\left(0,{\hkk{\rho}}(x_1)\right)}({\sig}) J_{\frac{n-3}{2}}(|\xi'| {\sig}) \ \dx{\sig} \ \dx x_1 \\
      &= \ \frac{\ome_{n-1}}{(2\pi)^{\frac n2}}\frac{1}{|\xi'|^{\frac{n-1}{2}}}
      \int_{\spt {\hkk{\rho}}} e^{-i x_1\xi_1} {\hkk{\rho}}(x_1)^{\frac{n-1}{2}}
      J_{\frac{n-1}{2}}(|\xi'| {\hkk{\rho}}(x_1)) \dx x_1.
    \end{split}
  \end{equation}
  Here,
  $J_{\alpha} = \sum_{k=0}^\infty (-1)^k(k!\ \Gam(k+1+\alp))^{-1}
  (\frac{z}{2})^{2k+\alp}$ is the Bessel function of first kind of order
  $\alpha \in \frac 12(\N - 1)$. With this series representation \eqref{app-prop-ft-n} follows. The bound
    \eqref{app-prop-ft-n-decay} follows from \eqref{app-prop-ft-dim-n} by applying the estimate $J_\alp(t) \lesssim t^{-\frac 12}$ for $\alp=\frac{n-1}{2} > 0$.
\end{proof}

\begin{proposition}[Recovery sequence] \label{prp-upper} %
  For any $A\in \AA_0^{(n)}$, $n \in \{2,3 \}$, there is a sequence
  $u_\eps\in\AA$ such that \eqref{top-conv} holds and
  $\limsup_{\eps\to 0} E_\eps^{(n)}[u_\eps] \leq E_0^{(n)}[A]$.
\end{proposition}
\begin{proof}
  By approximation we assume $A \in C_c^\infty(\R)$, the result in the general
  case then follows by taking a diagonal sequence.  We choose the constant
  recovery sequence $u_\eps\ := \ u \in \AA$ where
  $u(x) := \chi_{(0,\hkk{\rho}(x_1))}(|x'|)$ with $\hkk{\rho}$ given by
  $A(x_1) = \ome_{n-1} \hkk{\rho}^{n-1}(x_1)$.  The perimeter terms are easily
  estimated:
  \begin{align} \label{calc_Per_2dim_limsup} %
    P_\eps^{(n)}[u] %
    &= \int_{\R^2} \sqrt{|\nabla' u(x)|^2 + \eps^2 |\partial_1 u(x)|^2 } \dx x  \\ %
    &\leq %
    \begin{TC}
      \displaystyle 2|\dsk{\{A>0\}}| + C \eps \Vert A\Vert_{W^{1,1}(\R)} & \text{for $n=2$,}\\
      \displaystyle \sqrt{\pi} \int_\R \sqrt{{A(x_1)}} \dx x_1 + C \eps &
      \text{for $n=3$.}
    \end{TC}
  \end{align}
  For the estimate of the nonlocal term $N_\eps^{(n)}[u_\eps]$, we recall that
  according to Lemma \ref{lem-app-ft} the Fourier transform of $u$ is given by
  \begin{align}\label{app-prop-ft_n}
    \widehat u(\xi) \ %
    &= \ \frac{\ome_{n-1}}{(2\pi)^{\frac n2}}\sum_{k=0}^\infty
      \frac{(-1)^k |\xi'|^{2k}}{k!\Gam(k+1+\frac{n-1}{2})2^{2k+\frac{n-1}{2}}}
      {S_k(\xi_1)\, , }
  \end{align}
  where $S_k \hkk{:=} (A^{\frac{2k}{n-1}+1})^\wedge$ are Schwartz functions. In
  particular, $S_0 = \widehat{A}$.  We decompose the Fourier space into two
  domains, i.e. $\R^n = \DD \cup \DD^c$, where
  $\DD\ := \ \{ (\xi_1,\xi')\in\R^n \sd |\xi'|\geq \sig\ := \ 1 + |\xi_1|^{2}
  \}$ and we claim that the contribution of this set to the nonlocal energy is
  zero for $\eps\to 0$.  From \eqref{app-prop-ft-n-decay} in Lemma
  \ref{lem-app-ft} we know that
  $|\widehat{u}(\xi)|^2\lesssim \frac{1}{|\xi'|^n}$ since $A$ is bounded. With
  $m_\eps(\xi) := \frac{\xi_1^2}{\eps^2\xi_1^2 + |\xi'|^2}$ this implies
  \begin{align*}
    \int_{\DD} m_\eps(\xi) |\widehat{u}(\xi)|^2 \dx\xi \ %
    &\ \lesssim \  \ \int_\R \int_{\sig}^\infty \frac{\xi_1^2}{\eps^2\xi_1^2 + |\xi'|^2} \frac{1}{|\xi'|^n} \dx\xi' \dx\xi_1 \ %
      \ \lesssim \ \ \int_\R \frac{\xi_1^2}{\sig^{3}} \dx\xi_1.
  \end{align*}
    Since we chose
  $\sig = 1+|\xi_1|^{2}$, the term $\xi_1^2 \sig^{-3}$ is integrable and thus
  \begin{align} \label{calc_ub_3d_domain1} %
    \gam_n(\eps)\int_{\DD} m_\eps(\xi) |\widehat{u}(\xi)|^2 \dx\xi \ %
    \lesssim\ \gam_n(\eps) \, .
  \end{align}
  The second domain is the set $\R^n\setminus \DD$ .  In this region we can
  neglect higher order terms in $|\xi'|$.  For $k\geq 1$ we have by
  \eqref{calc_inner_int_expansion} of Lemma \ref{lem-intform} for $\beta=\sig$
  \begin{align} \label{calc_ub_2d_domain2a} %
  \int_{|\xi'|\leq \sig} m_\eps(\xi) |\xi'|^{2k} \dx\xi' \ %
    &\lesssim \ |\xi_1|^2 \sig^{2k+n-3} + \eps |\xi_1|^3 \sig^{2k+n-4} \ .
  \end{align}
\dsk{
With \eqref{app-prop-ft_n} and since $A$ is smooth, \eqref{calc_ub_2d_domain2a} implies
\begin{align} \label{calc_ub_2d_domain2aa}
\gamma_n(\eps)\int_{\DD^c} m_\eps(\xi) |\widehat{u}(\xi)|^2 \dx\xi \ %
 &\ \leq \ \frac{\gam_n(\eps)}{(2\pi)^{n-1}}
    \int_{\DD^c} m_\eps(\xi) |\widehat{A}(\xi_1)|^2 \dx\xi  + \gamma_n(\eps) C_A\, .
\end{align}
}Evoking Lemma \ref{lem-intform} once again with $\beta=1+\xi_1^2$ we
  get by \eqref{calc_inner_int} for $n = 2$, resp. $n=3$
  \begin{align} \label{calc_ub_2d_domain2b} %
    \frac{\gam_n(\eps)}{(2\pi)^{n-1}}
    \int_{\DD^c} m_\eps(\xi) |\widehat{A}(\xi_1)|^2 \dx\xi %
    &= \begin{cases} %
      \displaystyle \frac 1\pi \displaystyle \int_\R |\xi_1|
      \arctan\Big(\tfrac{1+\xi_1^2}{\eps|\xi_1|} \Big)
      |\widehat{A}(\xi_1)|^2 \dx\xi_1, \vspace{0.3ex}\\
      \displaystyle \ \frac{1}{28|\ln\eps|} \int_{\R} \xi_1^2\ln\Big( 1 +
      \Big|\tfrac{ 1+\xi_1^2 }{\eps|\xi_1|}\Big|^2 \Big) |\widehat{A}(\xi_1)|^2 \dx\xi_1.
      \end{cases}
  \end{align}
  \dsk{ Passing to the limit $\epsilon\rightarrow 0$ in the perimeter estimate \eqref{calc_Per_2dim_limsup} and the estimates for the nonlocal term \eqref{calc_ub_3d_domain1}, \eqref{calc_ub_2d_domain2aa} and \eqref{calc_ub_2d_domain2b} yields the assertion for $A_\eps\in C_c^\infty(\R)$. }
\end{proof}

\subsection{Proof of Theorem \ref{thm-limit}} \label{sus-limit} %

In this section, we calculate the solutions of the limit
  problems (see Fig. \ref{fig-min}), thus giving the proof of Theorem
  \ref{thm-limit}.

\medskip

Both energies $E_0^{(n)}$, $n = 2, 3,$ from \eqref{E0-2}--\eqref{E0-3} are
invariant under translation.
In the following, we separately consider the
cases $n=2,3$ and thus omit the superscript $(n)$ for the sake of readability.

\hktwo{ \textit{The case $n=2$:} From the energy $E_0$ it directly follows that
  the \ds{positivity set} of any minimizing sequence must be bounded and
  connected.  Indeed, a monotonicity argument shows that, otherwise, shifting
  the separate components together decreases the energy (see Lemma
  \ref{app-lem-H12-shifting}). Among functions with fixed positivity set, the
  energy is strictly convex and hence there exists a unique minimizer among this
  class of functions.  We first consider the case when the positivity set is
  given by $I := (-1,1)$.  } We get
\begin{align} \label{HAS-const} %
  \begin{split}
    0 \ %
    &= \ \pi\frac{d}{d\eta}\bigg|_{\eta=0} \Vert
    A+\eta\hktwo{\psi}\Vert^2_{\hhom} \ %
    = \int_\R |\tau|\widehat{A}(\tau) \widehat{\hktwo{\psi}}(\tau) \dx \tau \ %
    = \int_\R H(A')(t) \hktwo{\psi}(t) \ \dx t
  \end{split}
\end{align} 
for all $\hktwo{\psi}\in C_c^\infty(I)$ with
$\int_{I} \hktwo{\psi}(t)\dx t = 0$. Here, $H$ is the Hilbert transform given by
\begin{align}
  H(f)(s) \ =\   %
  \frac{1}{\pi} \ \mathrm{p.v.}\int_{\R} \frac{f(t)}{s-t} \dx t.
\end{align}
Identity \eqref{HAS-const} implies that $H(A') \hk{\equiv C_1}$ on $I$ \hk{for
  some constant $C_1 \in \R$} where $C_1$ is determined by the volume
constraint.  For fixed $I$ this uniquely determines $A$
\cite[eq.(45)]{Soehngen1939}, see also \cite[Thm. 1]{Rueland2019} \hkone{for the
  complementary case that $f$ and $H(f)$ are known on disjoint intervals}. We
claim that
\begin{align} \label{I-min} %
  A(t) \ %
  = \ \frac{2}{\pi}\sqrt{1 - t^2}\ \chi_{I}(t) \, .
\end{align}
Obviously, condition \eqref{A-unit} then holds. Furthermore,
$A'(t) = -\frac{2t}{\pi}\phi(t)$ where
$\phi(t) := \frac 1{\sqrt{1 - t^2}} \chi_{I}$. We recall that by
\cite[(4.111)]{King2009} we have
\begin{align}\label{lim-2d-eq-H}
  H(t \mapsto t \phi(t))(s) \ = \  - \frac{1}{\pi} \int_\R \phi(t)\dx t + s H(\phi)(s).
\end{align}
Since
$H\big(\phi\big)(s) \ = - (\sgn s) \frac 1{\sqrt{s^2-1}} \chi_{\R \BS I}(s)$
(cf. \cite[(21)]{Erdelyi1954}), \eqref{lim-2d-eq-H} implies
\begin{align*}
  H(A')(s) \
  &= \  \frac 2\pi \Big(\frac{1}{\pi}\int_{I} \frac{1}{\sqrt{1-t^2}} \dx t - \frac{s \sgn(s)}{\sqrt{s^2-1}}\chi_{\R \BS I}(s) \Big)  = \ \frac{2}{\pi} \Big( 1 - \frac{s \sgn(s)}{\sqrt{s^2-1}}\chi_{\R \BS I}(s) \Big),
\end{align*}
\hktwo{which confirms \eqref{I-min} for prescribed positivity set $I$}. By
translation invariance it is enough to consider intervals of the form
$I_{L_*^{(n)}} := (-L_*^{(n)},L_*^{(n)})$ for some $L_*^{(n)}>0$. With the same
calculation as before we see that the minimizer among the class of functions
with positivity set $I_{L_*}$ is given by $A_{L_*}(L_* t)=\frac{1}{L_*}A(t)$
where $A$ is given by \eqref{I-min}. We have
$\widehat{A_{L_*}}(\tau) = \frac{\sqrt{2}}{\sqrt{\pi}
  L_*|\tau|}J_1({L_*}|\tau|)$ \cite[p.11 (8)]{Erdelyi1954}, where $J_1$ is the
Bessel function of the first kind.  In particular,
$\NNN{A_{L_*}}{\hhom}^2 = \frac 2{\pi{L_*}^{2}}$
(cf. \cite[703]{Gradshteyn2000}) and hence
$E_0^{(2)}[A_{L_*}] = 4{L_*} - \frac{1}{\pi {L_*}^2}$.  Minimizing in ${L_*}$ we
obtain ${L_*} = (\frac{2}{\pi})^{\frac 13}$ and hence
$A_{L_*}(t) = (\frac{2}{\pi})^{\frac23}\sqrt{1 -
  \frac{t}{L_*}}\chi_I(\frac{t}{L_*})$.  Comparing this to \eqref{thm-limit-def-A-R}
we find $R_* = (\frac{1}{2\pi^2})^{1/3}$. \hk{The above calculation also shows
  that the interval length $L_*$ is uniquely determined by the minimization of
  the energy which --- together with the fact that we have uniqueness among
  functions with fixed \ds{positivity set} --- with the previous arguments shows
  that the problem admits a unique minimizer (up to translation). }

\medskip

\textit{The case $n=3$:} We first consider the minimization problem for
  symmetrically decreasing functions with support on some fixed interval
  $[-L_*,L_*]$. By the direct method of the calculus of variations there exists
  a minimizer $A$ in this class of functions. In terms of $R := R^{(3)}$, defined as in \eqref{thm-limit-def-A-R}, the
  Euler--Lagrange equation takes the form
   \begin{align}\label{lim-ELG-3D}
     \frac{R_*^3}{L_*^2} (R^2)''(t) \ = \ \frac {7}{R(t)} -
     2\bet_0 \qquad\qquad \text{in $[-1,0]$}
   \end{align}
   for some Lagrange multiplier $\bet_0 \hk{>} 0$ (the positivity follows from a
   simple rescaling argument) and $R_*>0$ satisfying $\pi R_*^2=A(0)$.  Multiplying \eqref{lim-ELG-3D} by
   $(R^2)'= 2RR'$, integrating and since $R' \geq 0$ in $[-1,0]$ we get
 \begin{align} \label{lim-ELEdrho} %
   R^2 (R')^2 \ %
   = \ \frac{L_*^2}{R_*^3} \Big( 7 R - \bet_0 R^2 + \frac{\alpha_0}{\beta_0}
   \Big) \qquad\qquad \text{for $t \in [-1,0]$ }
 \end{align}
 for some $\alpha_0 \in \R$ with
   $(RR')^2(-1) = \frac{L_*^2}{R_*^3}\frac{\alpha_0}{\beta_0}$. Since
   $R'(0) = 0$ and $R(0) = 1$ in \eqref{lim-ELEdrho} we get
   $\alpha_0 = \bet_0(\bet_0 - 7)$.  \DETAILS{
     \begin{align} \label{lim-ELErho} %
     \int_{-1}^t dt \ %
     = \ \frac{R_*^{3/2}}{L_*} \int_0^\rho \frac {R}{\sqrt{\bet_0 (1-R^2) - 7(1-R)}} \ dR%
     \qquad\qquad \text{for    $t \in [-1,0]$. }
     \end{align}}
   By explicit integration with
   $g(\rho) := 7\beta_0(\rho-1) -
   \beta_0^2(\rho^2-1)$ we get 
\begin{align}\label{lim-ELG-3D-impl-def-rho}
  t + 1 
   &=  \frac{R_*^{\frac 32}}{L_*\beta_0^{\frac 32}} \Big[\frac 7{2} \Big( \arctan \Big( \frac{7}{2\sqrt{\alpha_0}} \Big) -\arctan \Big( \frac{7 - 2\beta_0R}{2 \sqrt{g(R)}} \Big)\Big)  +  \left( \sqrt{\alpha_0} - \sqrt{g(R)}\right)\Big].
\end{align}
\DETAILS{ Im Knüpfer-Muratov paper kann man hier (wie in der alten (fehlerhaften) version, zb 6j) vom arctan zum arccos wechseln und das verinfacht die formel erheblich. 
in unserem Fall funktioniert das nicht so wirklich. in 6j klappt das weil $C_0=\alpha/\beta=0$, bei eurem paper sind die Konstanten etwas schöner (eine $1$ statt der $7$ und $\beta=1$)
\begin{align*}
\arctan\left( \frac{7 - 2\beta_0R}{2 \sqrt{\alpha_0 + 7\beta_0R - \beta_0^2R^2}} \right)
\ &= \
\arccos\left( \frac{2\sqrt{7\beta(r-1) - \beta^2(r^2-1)}}{4\beta^2-28\beta+49} \right) \\
\ &= \
\arccos\left( \frac{2\sqrt{\beta(r-1)(7 - \beta(r+1))}}{4\beta^2-28\beta+49} \right)
\end{align*}
Another difference with KM: \cite{KnuepferMuratov-2011} is stated a priori on a
bounded interval $(0,h)$ i.e. no problems with existence of minimizers, (we use
sequence $L_k$) since no volume constraint ($\beta=1$) no vanishing mass at
infinity } It remains to determine
the parameters $\beta_0, R_*, L_*$.  Setting $t=0$ in
\eqref{lim-ELG-3D-impl-def-rho} and since $R(0)=1$ we obtain
\begin{align} \label{lim-eq-abcd-2} L_* \ &= \
  \Big(\frac{R_*}{\beta_0}\Big)^{\frac 32} \Big[ \frac 72 \big( \arctan \big(
  \tfrac{7}{2\sqrt{\alp_0}} \big) - \tfrac {\pi}{2} \big) + \sqrt{\alp_0} \,
  \Big].
\end{align}
The condition $\pi R_*^2 L_* \int_\R R^2\dx t=1$ reads
\begin{align} \label{lim-eq-abcd-3}
  R_*^{-\frac 72}
  &= \ 2\pi \int_0^{1} \frac {R^3}{\sqrt{\frac{\alpha_0}{\beta_0} + 7R - \beta_0 R^2}} \ dR \nonumber \\
  \ &= \ \frac \pi{24\beta_0^{\frac72} } \Big(-1470 \sqrt{\alp_0} + 32 \alp_0^{3/2}
      + (252 \alp_0 + 5145)  \big( \arctan \big( \tfrac{7}{2\sqrt{\alp_0}}\big) -\tfrac \pi 2  \big) \Big).
\end{align}
Since the minimizer is also optimal with respect to volume preserving rescalings
of the form $R_\eps = \frac 1{\eps^{\frac 12}} R(\frac t\eps)$ we get the
additional condition
$0 = \pi R_* L_* \int R \dx t - \frac {6\pi}{7} \frac{R_*^4}{L_*} \int R^2
(R')^2 \dx t$.  Integrating \eqref{lim-ELEdrho} we can replace
$\int R^2 (R')^2 \dx t$ and an explicit integration of $\int R \dx t$ yields
\begin{align}\label{lim-eq-abcd-4}
  0 \ = \  
 R_*^\frac{7}{2}\left(
 -\frac {105}{2}\sqrt{\alpha_0}  + 5 \left(\alpha_0 + \tfrac{147}{4}\right)\Big(\tfrac {\pi}{2} - \arctan(\tfrac {7}{2 \sqrt{\alpha_0} }) \Big) \right)
  - \frac{12}{7}R_*^2 L_* \alpha_0 \beta_0^{3/2}  + \frac{6}{7\pi }\beta_0^{7/2}.
\end{align}
Inserting \eqref{lim-eq-abcd-2} and \eqref{lim-eq-abcd-3} into \eqref{lim-eq-abcd-4} we then obtain the unique solution $\alpha_0 \approx 104.332$ for $\alp_0$. Substituting this into \eqref{lim-eq-abcd-2} and \eqref{lim-eq-abcd-3}
we obtain $\beta_0 \approx 14.297$, $R_* \approx 1.511$ and $L_* \approx 0.202$. 
From \eqref{lim-ELEdrho} it follows that $R' \neq 0$ for
  $t \in [-1,1] \BS \{0\}$.

\DETAILS{
from optimality $L_*$
\begin{align*}
0 = \pi R_* L_* \int R - \frac{6\pi}{7} \frac{R_*^4}{L_*} \int R^2(R')^2
\end{align*}

from ELG
\begin{align*}
\int R^2 (R')^2 
= \frac{L^2}{R_*^3}\left( 7\int R - \beta\int R^2 + \int \frac{\alpha}{\beta} \right)
\end{align*}
\begin{align*}
\int \frac{\alpha}{\beta} 
&= 2 \frac{\alpha}{\beta}  \\
\beta\int R^2
&= \frac{\beta}{\pi R_*^2 L_*} \\
\int R
&= 2 \frac{R^\frac{3}{2}}{L}\int R^2/\sqrt{...} dR = 2\frac{R^{3/2}}{L}\left( \frac{4\alpha + 147}{8 \beta^{5/2}}(-\frac\pi 2 + \arctan(7/(2\sqrt{\alpha}))) + \frac{21 \sqrt{\alpha/\beta}}{4\beta^2} \right)
\end{align*}
Hence
\begin{align*}
0
&= (\pi R_* L_* - \frac{6\pi}{7}\frac{R_*^4}{L_*}\frac{L_*^2}{R_*^3}7)\int R 
+ \frac{6\pi}{7}\frac{R_*^4}{L_*}\frac{L_*^2}{R_*^3}\frac{\beta}{\pi R_*^2 L_*}
- \frac{6\pi}{7}\frac{R_*^4}{L_*}\frac{L_*^2}{R_*^3} 2 \frac{\alpha}{\beta} \\
&= -10\pi R_*^{5/2}\left( \frac{4\alpha + 147}{8 \beta^{5/2}}(-\frac\pi 2 + \arctan(7/(2\sqrt{\alpha}))) + \frac{21 \sqrt{\alpha/\beta}}{4\beta^2} \right)
+ \frac{6}{7}\frac{\beta}{R_*}
- \frac{12\pi}{7}R_* L_*\frac{\alpha}{\beta} \\
\end{align*}
}
 
\DETAILS{Multiplying \eqref{lim-ELG-3D} by $2 A'$, integrating and using that
  $\hk{\rho'} \geq 0$ in $[-L_*,0]$ we get
   \begin{align}
     (A')^2 \ = \  4 \pi^2 C_0 + 28\pi^\frac32 \sqrt{A} -4\pi \lam_0 A  \qquad %
     \text{for $t \in [-L_*,0]$. }
   \end{align}
   with $\lam_0 = \frac{\lam}{2\pi}$.  We introduce the radius by
   $A = \pi \rho^2$, i.e. $A' = 2\pi \rho \rho'$. Then
   \begin{align}
     E_0^{(3)}[\rho] \ %
     &= \ 2 \pi \int_\R \ \rho(t)  \dx t \ +\  \frac{2\pi}{7}\  \int_{\R} \rho'(t)^2 \rho^2(t) \dx t =: P + N, \label{E0-3rho}
   \end{align}
   and the ELE is
   $4\pi^2 \rho^2 (\rho')^2 = 4\pi^2 C_0 + 28\pi^2 \rho - 4\pi^2\lam_0 \rho^2$, i.e. 
   \begin{align} \label{ELErho} %
     \rho^2 (\rho')^2 = C_0 + 7 \rho - \lam_0 \rho^2. 
   \end{align}
   Since $\rho'(0) = 0$ and $\rho(0) = R_*$ we obtain
   \begin{align} \label{eq-co} %
     C_0 = \lam_0 R_*^2 - 7 R_*. 
   \end{align}
   In particular, $0 \leq C_0 < \infty$. We have $2$ remaining unknowns $L_*$
   and $\lam$.  By explicit integration this implies (wolfram alpha)
   \begin{align}
     \hspace{2ex} & \hspace{-2ex} %
                    t + L_* \ = \ \int_{-L_*}^{t} \ dt \ %
                    = \ \int_0^R \frac \rho {\sqrt{C_0 + 7 \rho -\lam_0 \rho^2 }} \ d\rho \\ %
     &= \ \Big[ - \frac 7{2\lam_0^{3/2}} \arctan \Big( \frac{7 - 2\lam_0\rho}{2\sqrt{\lam_0} \sqrt{C_0 + 7\rho-\lam_0\rho^2}} \Big) - \frac 1{\lam_0} \sqrt{C_0 + 7 \rho-\lam_0\rho^2}  \Big]_0^R \\
                  &= \  \frac 7{2\lam_0^{3/2}} \arctan \Big( \frac{7}{\sqrt{4\lam_0C_0}} \Big) + \frac{\sqrt{C_0}}{\lam_0} \\
                  &\qquad - \frac 7{2\lam_0^{3/2}} \arctan \Big( \frac{7 - 2\lam_0\rho}{2\sqrt{\lam_0} \sqrt{C_0 + 7\rho - \lam_0\rho^2}} \Big) - \frac 1{\lam_0} \sqrt{C_0 + 7\rho - \lam_0\rho^2}   %
   \end{align}
   If we set $t = 0$ (and hence $\rho = R_*$) above we get
   \begin{align}
     L_* \ %
     &= \ \frac 7{2\lam_0^{3/2}} \arctan \Big( \frac{7}{\sqrt{4\lam_0C_0}} \Big) + \frac{\sqrt{C_0}}{\lam_0}  - \frac {7\pi}{4\lam_0^{3/2}}.  %
   \end{align}
   The remaining $2$ unknowns $L_*$ and $\lam_0$ are determined by
   $\pi \int \rho^2 = 1$, and the optimality of the interval length.  The volume
   condition yields
   \begin{align} \notag \frac 1{2\pi} \ %
     &= \int_{-L_*}^0 \rho^2 \ %
       = \ \int_0^{R_*} \rho^2 \frac{dt}{d\rho} d\rho \ %
       = \ \int_0^{R_*} \frac {\rho^3}{\sqrt{C_0 + 7\rho - \lam_0 \rho^2}} \ d\rho \\
     &= \ \Big[\frac{8 C_0 \lam_0^2 \rho^2 - 182 C_0 \lam_0 \rho - 735 C_0 + 8 \lam_0^3 \rho^4 + 14 \lam_0^2 \rho^3 + 245 \lam_0 \rho^2 - 5145 \rho - 16 \lam_0 C_0^2 }{24 \lam_0^3\sqrt{C_0 + 7\rho - \lam_0 \rho^2} } \notag \\
     &\qquad - \frac {84 \lam_0 C_0 + 1715}{16 \lam_0^{7/2}} \arctan \Big( \frac{7-2 \lam_0 \rho}{2\sqrt{\lam_0}\sqrt{C_0 + 7\rho - \lam_0 \rho^2}}\Big)\Big]_0^{R_*} \notag \\
     &= \ \frac{735 C_0 + 16 \lam_0 C_0^2 }{24 \lam_0^3\sqrt{C_0} }  + \frac {84 \lam_0 C_0 + 1715}{16 \lam_0^{7/2}} \Big( -\frac \pi 2 + \arctan \Big( \frac{7}{2\sqrt{\lam_0 C_0 }} \Big) \Big). \notag \\
     &+ \ \frac{8 C_0 \lam_0^2 R_*^2 - 182 C_0 \lam_0 R_* - 735 C_0 + 8 \lam_0^3
       R_*^4 + 14 \lam_0^2 R_*^3 + 245 \lam_0 R_*^2 - 5145 R_* - 16 \lam_0 C_0^2
       }{24 \lam_0^3\sqrt{C_0 + 7R_* - \lam_0 R_*^2} }. \notag \\
  &= \frac{735 C_0 + 16 \lam_0 C_0^2 }{24 \lam_0^3\sqrt{C_0} }  + \frac {84 \lam_0 C_0 + 1715}{16 \lam_0^{7/2}} \Big( -\frac \pi 2 + \arctan \Big( \frac{7}{2\sqrt{\lam_0 C_0 }} \Big) \Big) . \notag \\
     &+  \frac{-\sqrt{C_0 + 7R_* - \lam_0 R_*^2}(16C_0\lambda_0 + 8\lambda_0^2R_*^2 + 70\lambda_0 R_* + 735)}{24 \lam_0^3 }  
   \nonumber \\
 &= \frac{735 C_0 + 16 \lam_0 C_0^2 }{24 \lam_0^3\sqrt{C_0} }  + \frac {84 \lam_0 C_0 + 1715}{16 \lam_0^{7/2}} \Big( -\frac \pi 2 + \arctan \Big( \frac{7}{2\sqrt{\lam_0 C_0 }} \Big) \Big) . \notag  
   \end{align}
   by \eqref{eq-co} for $C_0$.  We consider
   $A_\eps = \frac 1\eps A(\frac x\eps)$, resp.
   $\rho_\eps = \frac 1{\eps^{\frac 12}} \rho(\frac x\eps)$. For the minimizer
   we have
   \begin{align} \label{intopt} %
     0 \ &= \ \frac{d}{d\eps} E(A_\eps)\bigg|_{\eps = 1} \ %
 \ = \ \frac{d}{d\eps} E(A_\eps)\bigg|_{\eps = 1} 2\pi\sqrt{\eps}\int\rho + \frac{2\pi}{7}\epsilon^{-3}\int (\rho\rho')^2 \\       
   &= \ \pi \int \rho \ dx - \frac {6\pi}{7} \int \rho^2 (\rho')^2
           \ dx \ %
           =: \ \pi^{3/2} P - \frac {3\pi}{14} N.
   \end{align}
   Integrating the ELE we get $N = 2C_0 L_* + 7 P - \frac{\lam_0}{\pi}$. Using
   this to eliminate $N$ we have
   \begin{align}
     0 \ = \ \pi P -  \frac{12\pi}{6} C_0 L_* - 6\pi P +\frac{6\pi}{7}\lambda_0 .
   \end{align}
   We can calculate $P$ explicitly for the solution, i.e.
   \begin{align*}
    \frac12 P := \int_{-L_*}^0 \rho dt  \
     &= \ \int_0^{R_*} \rho\frac{dt}{d\rho} d\rho \ %
       = \ \int_0^{R_*} \frac {\rho^2}{\sqrt{C_0 + 7\rho - \lam_0 \rho^2}} \ d\rho \\
     & = \ \Big[ -\frac {2 \lam_0 \rho + 21}{4 \lam_0^2}  \sqrt{C_0 + 7 \rho - \lam_0 \rho^2} %
       - \frac {4 \lam_0 C_0 + 147}{8\lam_0^{5/2}}  \arctan(\frac {7-2\lam_0\rho}{2 \sqrt{\lam_0} \sqrt{C_0 + 7 \rho - \lam_0 \rho^2} })\Big]^{R_*}_0 \\
     & = \ \frac {21}{4 \lam_0^2}  \sqrt{C_0} %
       + \frac {4 \lam_0 C_0 + 147}{8\lam_0^{5/2}}  \Big(\arctan(\frac {7}{2 \sqrt{\lam_0 C_0} }) %
       -  \frac {\pi}{2} \Big).
   \end{align*}

   For the
   determination of the parameters $C_0$, $\lam_0$, $R_*$, $L_*$ we have
   \begin{align*} %
     C_0 &= \lam_0 R_*^2 - 7 R_*. \\
     L_* \ %
         &= \ \frac 7{2\lam_0^{3/2}} \arctan \Big( \frac{7}{2\sqrt{\lam_0C_0}}
           \Big) + \frac{\sqrt{C_0}}{\lam_0} - \frac {7\pi}{4\lam_0^{3/2}}, \\ %
    \frac{1}{2\pi} \ %
         &= \ \frac{735 C_0 + 16 \lam_0 C_0^2 }{24 \lam_0^3\sqrt{C_0} }  + \frac {84 \lam_0 C_0 + 1715}{16 \lam_0^{7/2}} \Big( -\frac \pi 2 + \arctan \Big( \frac{7}{2\sqrt{\lam_0}\sqrt{C_0} }\Big) \Big). \notag \\
     0 \ &= \ -5\pi P  -  \frac{12\pi}{7}C_0 L_* + \frac{6\pi}{7}\lambda_0 \, .
   \end{align*}
   where $P$ is given above. We can rewrite the system for
   $C_0,\lambda_0,R_*,L_*$ in terms of
   $\alpha_0,\beta_0,\gamma_0 := L_* \lambda_0^\frac32, \d_0 :=
   \lambda_0^{\frac72}$,
   \begin{align*} %
     \alp_0 &= \bet_0^2 - 7 \bet_0. \\
     \gam_0 \ %
            &= \  \frac 72 \Big( \arctan \Big( \frac{7}{2\sqrt{\alp_0}}
              \Big) - \frac {\pi}{2} \Big) + \sqrt{\alp_0} , \\ %
     \frac{\d_0}{2\pi} \ %
            &= \  \frac{735 \alp_0 + 16 \alp_0^2 }{24\sqrt{\alp_0} }  + \frac {84 \alp_0 + 1715}{16} \Big( \arctan \Big( \frac{7}{2\sqrt{\alp_0}}\Big) -\frac \pi 2  \Big) \notag \\
     0 \ &= \  -5\pi\left( \frac {21}{2}\sqrt{\alpha_0} %
           + (\alpha_0 + \frac{147}{4})\Big(\arctan(\frac {7}{2 \sqrt{\alpha_0} }) -  \frac {\pi}{2} \Big) \right) - \frac{12\pi}{7} \alpha_0 \gam_0 + \frac{6\pi}{7}\d_0 
   \end{align*}
   If we insert the second and third equation into the fourth equation we get an
   equation for $\alp_0$, i.e. with $K_0 := \arctan(\frac {7}{2 \sqrt{\alpha_0} }) -  \frac {\pi}{2}$ we have
   \begin{align*} 
     0 \ &= \  -5 \left( \frac {21}{2}\sqrt{\alpha_0} %
           + (\alpha_0 + \frac{147}{4}) K_0 \right) -  6 \alpha_0 K_0 - \frac{12}{7} \alpha_0^{3/2}) \\
         &\qquad + \frac{3\pi}{7}\Big[ \frac{735}6 \sqrt{\alp_0} + \frac 83 \alp_0^{3/2}  + 21 \alp_0 K_0 + \frac{1715}{4} K_0 \Big]
   \end{align*}
   \dsc{Als Lösungen für $\alpha_0$ erhalte ich $0$ und $\approx 18.87999467$ Der
     Plot legt nahe dass das die einzigen Lösungen sind.

Für $\alpha_0=0$ ist natürlich $R=0=L$ und daher falsch.
Für die Andere Lösung erhalte ich
\begin{align*}
alpha &=  18.87999467 \\
beta  &=  9.07942602 \\
gamma &=  1.220630256094004 \\
delta &=  22.537603973793434
\end{align*}
und somit
\begin{align*}
lambda &=  2.435258509158324 \\
C_0 &=  7.752768175944211 \\
R_* &=  3.72832124 \\
L_* &=  0.3211933362545995
\end{align*}
}
} 

\medskip

In view of \cite[Chapt. 3]{Lieb2001} and the Pólya-Szeg\H{o} inequality
\cite[Theorem 1.1]{Brothers1988}, see also \cite[Ch. III]{Polya1951}), symmetric
rearrangement does not increase the energy. In fact, since our minimizer
satisfies $|\{x\sd A'(x)=0\}\cap\supp(A)|=0$, any configuration whose symmetric
rearrangement is given by this minimizer has strictly larger energy
\cite[Thm. 1]{Ferone2004}. \hk{ Hence, our arguments show that within the class
  of functions with bounded support, the unique minimizer, up to translation, is
  given by $A$. Since the energy of any function $u \in \AA$ can be approximated
  by the energy of functions with bounded support, this shows that $A$ is a
  minimizer of the energy within the class of functions $\AA$ and that this
  minimizer is unique, up to translation, within the class of functions with
  bounded support. Furthermore, any minimizer satisfies the Euler-Lagrange
  equation. In turn a simple calculation shows that solutions of the
  Euler-Lagrange equation with finite mass necessarily have compact
  support. Indeed, this follows since the ordinary differential inequality
  $A' \leq - C \sqrt[4]{A}$ does not allow for solutions with finite mass and
  unbounded support as a straightforward calculation shows.  }
 
\appendix

\section{} 

\hk{In the appendix we give the proof of Theorem \ref{thm-scaling} \dsk{ and Theorem \ref{thm-compactness}} for $n = 4$ \dsk{as well as} some auxiliary estimate for the $H^{\frac 12}$--norm used in the
  Proof of Theorem \ref{thm-limit}. } 
\begin{lemma}[Compactness for $n = 4$] \label{lem-compactness} %
  Let $n = 4$. Then for any sequence $u_\eps\in \AA$ with
  $\spt u_\eps \subset B_\rho(0)$ for some $\rho>0$,
  $E_\eps^{(n)}[u_\eps] \lesssim 1$ and $0\leq s < \frac{1}{22}$, there is a
  subsequence (not relabelled) and a function $A\in\AA_0^{(n)}$ such that
  \begin{align} \label{top-conv-nd} %
    \A[u_\eps] \ \wto \ A \quad\text{ in } H^s(\R) \quad \text{as $\eps \to 0$.}
  \end{align}
\end{lemma}
Although the proof is only concerned with $n=4$, we will carry out steps in full
generality, i.e.\ $n\geq 4$ whenever possible since this gives a hint for
  the structure of the estimates and the loss of regularity for $n \geq 4$.

\begin{proof}
  Using estimate \eqref{calc_inner_int_expansion} of Lemma \ref{lem-intform} for $k = 0$
    and $\beta = \eta|\widehat{A_\eps}(\xi_1)|^\frac12$ we get for $n\geq 4$
  and
  ${q_{\eps,\eta}}(\xi_1) := \frac{\eta|\widehat{A_\eps}|^\frac12}{\eps|\xi_1|} \geq 1$
  \begin{align} \label{calc_inner_int_ngeq4} %
  \gam_n(\eps) \int_{|\xi'|\leq \eta|\widehat{A_\eps}(\xi_1)|^\frac12}
    \frac{\xi_1^2}{\eps^2\xi_1^2 + |\xi'|^2} \dx\xi' %
    = 
   \frac{\ome_{n-2}}{n-3} \xi_1^{2} \eta^{n-3}  |\hat{A_\eps}(\xi_1)|^{\frac{n-3}{2}} + R\, ,
  \end{align}
  where $R=\eps \, O(\eta^{n-4}|\xi_1|^3 |\hat{A_\eps}(\xi_1)|^{\frac{n-4}{2}})$.  This
  implies that, by Lemma \ref{lem-N-low}
  \begin{align}\label{N_ngeq4}
    N_\eps^{(n)}
    \ &\geq \ds{C_\eta} \ \int_{\R\cap \{{q_{\eps,\eta}}\geq 1\}} \xi_1^2 |\hat A_\eps(\xi_1)|^{\frac{n+1}{2}}\dx\xi_1.
  \end{align}
  for $\eps$ small enough. Our goal is to bound some $H^s-$norm of
  $A_\epsilon$. Setting $p = \frac{n+1}{4}$ and $q=\frac{n+1}{n-3}$ we have
  $p^{-1}+q^{-1}=1$ and for $\alpha>\frac{n-3}{n+1}$ it follows by Hölder's
  inequality
  \begin{align} \label{proof-cpt-N_geq4} %
    \hspace{6ex} & \hspace{-6ex} %
                   \int_{(-1,1)^c\cap \{{q_{\eps,\eta}}\geq 1\}} \xi_1^{2s} |\hat A_\epsilon(\xi_1)|^2 \dx \xi_1 \\
    &\leq \Big( \int_{(-1,1)^c\cap \{{q_{\eps,\eta}}\geq 1\}} \xi_1^{(2s+\alpha)p} |\hat A_\epsilon(\xi_1)|^{2p}
    \dx\xi_1 \Big)^{\frac1p} %
    \Big( \int_{(-1,1)^c} \xi_1^{-\alpha q} \dx\xi_1 \Big)^{\frac 1q}. \nonumber
  \end{align}
  By our choice of $p$ we have $2p = \frac{n+1}{2}$ and since
  $\alpha>\frac{n-3}{n+1}$ it holds that $\alpha q >1$, so the second integral
  converges.  To apply our estimate \eqref{N_ngeq4} we require
  $(2s+\alpha)p\leq 2$, which is equivalent to
  $2s < \frac{8}{n+1} - \frac{n-3}{n+1} = \frac{11 - n}{n+1}$.  This yields
  compactness of $(\hat{A_\eps}\chi_{\{{q_{\eps,\eta}}\geq 1\}})^\vee$ in $H^s$
  for some $s>0$ in all dimensions $4\leq n\leq 10$. For the complementary
  estimate, i.e.\ if ${q_{\eps,\eta}}(\xi_1)\leq 1$, we proceed as in the proof
  of Lemma \ref{lem-com1} for $n=2,3$ to find that
  $\int_{\eps^{-\gam}}^\infty |\xi_1|^{2s}|\widehat{A_\eps}|\dx\xi_1$ is
  bounded, provided $\gam\geq \frac{2}{1-2s}$.  Furthermore, since
  ${q_{\eps,\eta}}(\xi_1)\leq 1$, we have the lower bound
\begin{align*}
\gam_n(\eps) \int_{|\xi'|\leq \eta|\widehat{A_\eps}(\xi_1)|^\frac12} \frac{\xi_1^2}{\eps^2\xi_1^2 + |\xi'|^2} \dx\xi' %
\ \gtrsim \ \eps^{n-3}|\xi_1|^{n-1} \left( \frac{\eta|\widehat{A_\eps}|^\frac{1}{2}}{\eps|\xi_1|} \right)^{n-1}\, ,
\end{align*} 
from which it follows that 
\begin{align*}
\int_{\R\cap \{{q_{\eps,\eta}}\leq 1\}} |\widehat{A_\eps}|^\frac{n+3}{2} \dx\xi_1 \ \leq \ C_\eta \eps^2\, .
\end{align*}
We can estimate by Hölder's inequality
 \begin{align} \label{est-R_eps_c-nd}
 \NNN{\big(\widehat{A_\eps}\chi_{[0,\eps^{-\gam})\cap \{{q_{\eps,\eta}}\leq 1\}}\big)^\vee}{\dot H^{s}}^2 \ %
    &\lesssim \ \big(\sup_{\xi_1\in [0,\epsilon^{-\gamma})} |\xi_1|^{2s}\big) \NNN{\widehat{A_\eps}\chi_{\{{q_{\eps,\eta}}\geq 1\}}}{L^{\frac{n+3}{2}}}^2 \eps^{-\frac{n-1}{n+3}\gam} \ \\
    &\lesssim \ \eps^{\frac{8}{n+3} - \gam(\frac{n-1}{n+3} + 2s)}\ \lesssim \ 1\ ,
\end{align}
for $\gam\leq \frac{8}{2s(n+3)+n-1}$.  Both inequalities for $\gam$ yield
together that $s<\frac{5-n}{2(n+7)}$, in particular $s<\frac{1}{22}$ for $n=4$.
Note that this estimate does not provide us with control over an $H^s$ norm for
the part ${q_{\eps,\eta}}\geq 1$ with positive $s$ for dimensions $n\geq 5$. The
better regularity for the limit function $A\in\AA_0^{(n)}$ follows as in
Proposition \ref{prp-com23} by using \eqref{calc_inner_int_ngeq4} and Lemma
\ref{lem-weak_pointwise_conv}.  Since $\gamma_n(\eps)=1$ for $n\geq 4$, we can
take $S_\eps=R_\eps$ in the notation of Lemma \ref{lem-com1}.
\end{proof}

\begin{lemma} \label{lem-app-scaling}
The lower bound in Theorem \ref{thm-scaling} holds for $n=4$, the upper bound holds for all $n\geq 4$.
\end{lemma}

\begin{proof}
Assume that the infimum in \eqref{thm-scaling-eq_inf} for
$n=4$ is zero. By \eqref{N_ngeq4} and \eqref{proof-cpt-N_geq4} (for $s=0$) this
would imply that $\Vert \hat{A_\eps}\Vert_{L^2}$ converges to zero, i.e.\
$\Vert A_\eps\Vert_{L^2}\rightarrow 0$ and since the support is bounded, also $\Vert A_\eps\Vert_{L^1}\rightarrow 0$. But this would contradict the assumption
that $\Vert A_\eps\Vert_{L^1}=1$. Hence, the infimum is strictly positive.

  \medskip
  
For the upper bound in \eqref{thm-scaling-eq_inf}, it is enough to construct a
  sequence with uniformly bounded energy. So let $n\geq 4$ and
  $A\in C_c^\infty(\R)$ with $\Vert A\Vert_{L^\infty}\leq 1$ and let $u$
  be given as in Lemma \ref{lem-app-ft}.  Then obviously the perimeter
  $P_\eps^{(n)}[u]$ is bounded.  For the nonlocal energy we can estimate as for
  \eqref{calc_ub_3d_domain1} for $\sig := 1+|\xi_1|^2$
  \begin{align*}
    \int_{\R} \int_{|\xi'|\geq \sig} \frac{\xi_1^2}{\eps^2\xi_1^2 + |\xi'|^2} |\hat{u}(\xi)|^2 \dx\xi' \dx\xi_1
    \ \lesssim \ 1\, .
  \end{align*}
 For the remaining part of the nonlocal energy, we note that by Lemma
  \ref{lem-app-ft} 
  \begin{align*} %
    |\hat u(\xi)|^2 
    \ &= \ \frac{\ome_{n-1}^2}{(2\pi)^n}\sum_{k=0}^\infty \frac{(-1)^k}{2^{2k+\frac{n-1}{2}}}  \sum_{\ell=0}^k \frac{\hkk{|\xi'|^{2k}} S_\ell(\xi_1) S_{k-\ell}(\xi_1)}{\ell!(k-\ell)!\Gam(\ell+1+\frac{n-1}{2})\Gam(k-\ell+1+\frac{n-1}{2})}\, ,
  \end{align*}
  where we used again the notation
  $S_\ell := (A^{\frac{2\ell}{n-1}+1})^{\wedge}$.  By Lemma \ref{lem-intform}
  for $\beta = \sig$ we get for small $\eps$
  \begin{align*}
    \int_{|\xi'|\leq \sig} \frac{\xi_1^2 |\xi'|^{2k}}{\eps^2\xi_1^2 + |\xi'|^2} \dx \xi'                                                                                            \ &= \ \frac{n-1}{n-3+2k}\:\ome_{n-1}\  \xi_1^{2} \sig^{n-3+2k} + O(\eps |\xi_1|^3 \sig^{n-4+2k})\, .
  \end{align*}
Combining \hkk{the above two identities} we can estimate
  \begin{align*}
    \int_{|\xi'|\leq \sig} \frac{\xi_1^2}{\eps^2\xi_1^2+|\xi'|^2} |\hat u(\xi)|^2 \dx\xi' 
\ &\sim \  \xi_1^2 \sig^{n-3} |\hat{u}(\xi_1,|\xi'|=\sig)|^2  \\
\ &\sim \ \xi_1^2 \sig^{-2} \Big| \int_{\spt(A)} e^{-i x\cdot\xi} |{A}(x_1)|^{\frac{1}{2}} J_{\frac{n-1}{2}}(\sig|{A}(x_1)|^\frac{1}{n-1})\dx x_1 \Big|^2\, ,
  \end{align*}
 where the last step is due to the calculation in Lemma \ref{lem-app-ft}.
  We decompose
  $\R\times\spt(A) = \{ (\xi_1,x_1)\sd \sig|{A}(x_1)|^\frac{1}{n-1}\geq 1 \}
  \cup \{ (\xi_1,x_1)\sd \sig|{A}(x_1)|^\frac{1}{n-1}\leq 1 \}$ and since for
  $\alpha\in \frac12\mathbb{N}$ we have $|J_\alpha(t)|\lesssim t^{-\frac12}$ for
  $t\geq 1$ and $t^{\alpha}$ for $t\leq 1$ we get
  \begin{align*}
    \int_\R &\int_{|\xi'|\leq \sig} \frac{\xi_1^2}{\eps^2\xi_1^2+|\xi'|^2} |\hat u(\xi)|^2 \dx\xi' \dx\xi_1 \\
 \ &\sim \ \int_\R \xi_1^2 \sig^{-2} \Big| \int_{\spt(A)} e^{-i x_1\cdot\xi_1} |A(x_1)|^\frac12 J_{\frac{n-1}{2}}(\sig|{A}(x_1)|^\frac{1}{n-1})\dx x_1 \Big|^2  \dx\xi_1 \\
    \ &\lesssim \ |\spt A|^2 \int_{\R} \xi_1^2 \sig^{-3}  \dx\xi_1  + \int_\R \xi_1^2 \sig^{n-3}|\widehat A(\xi_1)|^2\dx\xi_1 \, .
  \end{align*}
These integrals are finite since $\xi_1^2(1+\xi_1^2)^{-3}$ is integrable and $A$ is smooth.
\end{proof}
We note that the $\dot{H}^\frac12-$norm of $A$ can also be expressed (up to a
constant which only depends on $s$ and the space dimension, see
\cite[Prop. 3.4]{Nezza2012}) as
\begin{align*}
  \Vert A\Vert_{\dot{H}^\frac12(\R^n)}^2 \ = \ C_{s,n}\int_{\R^n}\int_{\R^n}
\frac{|A(x)-A(y)|^2}{|x-y|^{n+1}}\dx x\dx y
\end{align*}
With this identity we prove the following
auxiliary estimate that has been used in the proof of Theorem \ref{thm-limit}:
\begin{lemma}\label{app-lem-H12-shifting}
  Let $\phi,\psi \in H^\frac{1}{2}(\mathbb{R})$.  Assume that
    $A :=\phi+\psi$ with $\phi,\psi \geq 0$ with
    $\{\phi > 0\} \subset (-\infty,0]$ and $\{\psi>0\} \subset [\tau,\infty)$ for
    some $\tau \geq 0$.  Then
  \begin{align*}
    \Vert A\Vert_{\dot{H}^\frac12(\R)}^2
    \ > \ \Vert A_t\Vert_{\dot{H}^\frac12(\R)}^2\, ,
  \end{align*}
  where $A_t(x):= \phi(x)+\psi(x+t)$ for $t\in (0,\tau]$.
\end{lemma}
\begin{proof}
  A direct calculation shows that
\begin{align*}
  \frac{\dx}{\dx t} \int_\R \int_\R \frac{|A_t(x) - A_t(y)|^2}{|x-y|^{2}} \dx x\dx y
     &= \  - 4 \frac{\dx}{\dx t} \int_\R \int_\R \frac{\phi(x)\psi(y+t)}{|x-y|^{2}}\dx x\dx y\, ,
\end{align*} %
using that $\phi$ and $\psi$ have disjoint support for $t\leq \tau$.  Without
loss of generality we can assume $\tau = \dist(\{\phi > 0\},\{\psi > 0\}).$ With
the change of variables $z = y+t$ we have
\begin{align*}
  \frac{\dx}{\dx t} \int_\R \int_\R \frac{\phi(x)\psi(y+t)}{|x-y|^{2}}\dx x\dx y
  \ &= \ \frac{\dx}{\dx t} \int_\R \int_\R \frac{\phi(x)\psi(z)}{|x-z+t|^{2}}\dx x\dx z \\
  \ &= \ -2\int_{\supp(\psi)} \int_{\supp(\varphi)} (x-z+t)\frac{\phi(x)\psi(z)}{|x-z+t|^{4}}\dx x\dx z\, .
\end{align*}
Since $\phi,\psi\geq 0$ and $x-z\leq -\tau$, the above expression stays strictly positive as long as $t\in (0,\tau)$.
\end{proof}

\textbf{Acknowledgements:} \dstwo{We are grateful to the referees for their
  careful reading of the manuscript and their useful comments.}  The authors
thank B. Brietzke for valuable comments and discussions.  H.~Kn\"upfer was
partially supported by the German Research Foundation (DFG) by the project
\#392124319 and under Germany's Excellence Strategy – EXC-2181/1 – 390900948.

\bibliographystyle{abbrv} 

\begin{thebibliography}{10}

\bibitem{Agarwal2011}
S.~Agarwal, G.~Carbou, S.~Labb{\'{e}}, and C.~Prieur.
\newblock Control of a network of magnetic ellipsoidal samples.
\newblock {\em Mathematical Control {\&} Related Fields}, 1(2):129--147, 2011.

\bibitem{Aharoni1988}
A.~Aharoni.
\newblock Elongated single-domain ferromagnetic particles.
\newblock {\em Journal of Applied Physics}, 63(12):5879--5882, jun 1988.

\bibitem{Alama2020}
S.~Alama, L.~Bronsard, I.~Topaloglu, and A.~Zuniga.
\newblock A nonlocal isoperimetric problem with density perimeter.
\newblock {\em Calculus of Variations and Partial Differential Equations},
  60(1), nov 2020.

\bibitem{Alouges2008}
F.~Alouges and K.~Beauchard.
\newblock Magnetization switching on small ferromagnetic ellipsoidal samples.
\newblock {\em ESAIM: COCV}, 15(3):676--711, jul 2008.

\bibitem{AnzellottiBaldoVisintin-1991}
G.~Anzellotti, S.~Baldo, and A.~Visintin.
\newblock Asymptotic behavior of the {L}andau-{L}ifshitz model of
  ferromagnetism.
\newblock {\em Appl. Math. Optim.}, 23(2):171--192, 1991.

\bibitem{Bahiana1990}
M.~Bahiana and Y.~Oono.
\newblock Cell dynamical system approach to block copolymers.
\newblock {\em Phys. Rev. A}, 41(12):6763--6771, 1990.

\bibitem{Banerjee2001}
S.~Banerjee and M.~Widom.
\newblock Shapes and textures of ferromagnetic liquid droplets.
\newblock {\em Braz. J. Phys.}, 31(3):360--365, sep 2001.

\bibitem{BellaGZ-2015}
P.~Bella, M.~Goldman, and B.~Zwicknagl.
\newblock Study of island formation in epitaxially strained films on unbounded
  domains.
\newblock {\em Archive for Rational Mechanics and Analysis}, 218(1):163--217,
  apr 2015.

\bibitem{Berkovsky1985}
B.~Berkovsky and V.~Kalikmanov.
\newblock Topological instability of magnetic fluids.
\newblock {\em Journal de Physique Lettres}, 46(11):483--491, 1985.

\bibitem{Bonacini2014}
M.~Bonacini and R.~Cristoferi.
\newblock Local and global minimality results for a nonlocal isoperimetric
  problem on rn.
\newblock {\em SIAM J. Math. Anal.}, 46(4):2310--2349, 2014.

\bibitem{Bonacini2015}
M.~Bonacini, H.~Kn\"{u}pfer, and M.~R\"{o}ger.
\newblock Optimal distribution of oppositely charged phases: perfect screening
  and other properties.
\newblock {\em SIAM J. Math. Anal.}, 48(2):1128--1154, 2016.

\bibitem{Brancher1987}
J.~Brancher and D.~Zouaoui.
\newblock Equilibrium of a magnetic liquid drop.
\newblock {\em Journal of Magnetism and Magnetic Materials}, 65(2-3):311--314,
  mar 1987.

\bibitem{Brothers1988}
J.~Brothers and W.~Ziemer.
\newblock Minimal rearrangements of sobolev functions.
\newblock {\em Journal für die reine und angewandte Mathematik (Crelles
  Journal)}, 1988(384):153--179, feb 1988.

\bibitem{Brown1968}
W.~Brown.
\newblock The fundamental theorem of fine-ferromagnetic-particle theory.
\newblock {\em Journal of Applied Physics}, 39(2):993--994, feb 1968.

\bibitem{Brusentsov2001}
N.~A. Brusentsov, V.~Gogosov, T.~Brusentsova, A.~Sergeev, N.~Jurchenko, A.~A.
  Kuznetsov, O.~A. Kuznetsov, and L.~Shumakov.
\newblock Evaluation of ferromagnetic fluids and suspensions for the
  site-specific radiofrequency-induced hyperthermia of {MX}11 sarcoma cells in
  vitro.
\newblock {\em Journal of Magnetism and Magnetic Materials}, 225(1-2):113--117,
  jan 2001.

\bibitem{CandauTilh2021}
J.~Candau-Tilh and M.~Goldman.
\newblock Existence and stability results for an isoperimetric problem with a
  non-local interaction of wasserstein type.
\newblock {\em Preprint}, 2021.

\bibitem{Carrillo2020}
J.~Carrillo, J.~Mateu, and E.~al.
\newblock The equilibrium measure for an anisotropic nonlocal energy.
\newblock {\em Calc. Var. Partial Differential Equations}, to appear.

\bibitem{Carrillo2019}
J.~A. Carrillo, J.~Mateu, M.~G. Mora, L.~Rondi, L.~Scardia, and J.~Verdera.
\newblock The ellipse law: Kirchhoff meets dislocations.
\newblock {\em Comm. Math. Phys.}, 373(2):507--524, apr 2019.

\bibitem{ChoksiCKO-2008}
R.~Choksi, S.~Conti, R.~V. Kohn, and F.~Otto.
\newblock Ground state energy scaling laws during the onset and destruction of
  the intermediate state in a type i superconductor.
\newblock {\em Commun. Pure Appl. Math.}, 61(5):595--626, 2008.

\bibitem{Choksi2003}
R.~Choksi and X.~Ren.
\newblock On the derivation of a density functional theory for microphase
  separation of diblock copolymers.
\newblock {\em J. Stat. Phys.}, 113(1/2):151--176, 2003.

\bibitem{Clark2013}
N.~A. Clark.
\newblock Ferromagnetic ferrofluids.
\newblock {\em Nature}, 504(7479):229--230, dec 2013.

\bibitem{ContiGO-2015}
S.~Conti, A.~Garroni, and M.~Ortiz.
\newblock The line-tension approximation as the dilute limit of linear-elastic
  dislocations.
\newblock {\em Archive for Rational Mechanics and Analysis}, 218(2):699--755,
  may 2015.

\bibitem{ContiGOS-2018}
S.~Conti, M.~Goldman, F.~Otto, and S.~Serfaty.
\newblock A branched transport limit of the~ginzburg-landau functional.
\newblock {\em Journal de l'{\'{E}}cole polytechnique {\textemdash}
  Math{\'{e}}matiques}, 5:317--375, 2018.

\bibitem{ContiSchweizer-2006}
S.~Conti and B.~Schweizer.
\newblock Rigidity and gamma convergence for solid-solid phase transitions with
  {SO}(2) invariance.
\newblock {\em Comm. Pure Appl. Math.}, 59(6):830--868, 2006.

\bibitem{Demengel2012}
F.~Demengel and G.~Demengel.
\newblock {\em Functional Spaces for the Theory of Elliptic Partial
  Differential Equations}.
\newblock Springer London, 2012.

\bibitem{DeSimone2000}
A.~DeSimone, R.~V. Kohn, S.~Müller, and F.~Otto.
\newblock Magnetic microstructures - a paradigm of multiscale problems.
\newblock {\em ICIAM 99 (Edinburgh)}, pages 175--190, 2000.

\bibitem{Desimone2002}
A.~Desimone, R.~V. Kohn, S.~Müller, and F.~Otto.
\newblock A reduced theory for thin-film micromagnetics.
\newblock {\em Commun. Pure Appl. Math.}, 55(11):1408--1460, 2002.

\bibitem{Fratta2016}
G.~Di~Fratta.
\newblock The {N}ewtonian potential and the demagnetizing factors of the
  general ellipsoid.
\newblock {\em Proc. A.}, 472(2190):20160197, 7, 2016.

\bibitem{EleuteriLT-preprint2020}
M.~Eleuteri, L.~Lussardi, and A.~Torricelli.
\newblock Limits of non-local anisotropic perimeters.
\newblock {\em arXiv}, 2020.

\bibitem{Erdelyi1954}
A.~Erd\'{e}lyi, W.~Magnus, F.~Oberhettinger, and F.~G. Tricomi.
\newblock {\em Tables of integral transforms.}, volume~2.
\newblock McGraw-Hill Book Company, Inc., New York-Toronto-London, 1954.

\bibitem{Fabian1999}
K.~Fabian and A.~Hubert.
\newblock Shape-induced pseudo-single-domain remanence.
\newblock {\em Geophys. J. lnt.}, 138(3):717--726, sep 1999.

\bibitem{Ferone2004}
A.~Ferone and R.~Volpicelli.
\newblock Convex rearrangement: Equality cases in the {P}olya-{S}zegő
  inequality.
\newblock {\em Calculus of Variations}, 21(3), nov 2004.

\bibitem{Frank2015}
R.~Frank and E.~Lieb.
\newblock A compactness lemma and its application to the existence of
  minimizers for the liquid drop model.
\newblock {\em SIAM J. Math. Anal.}, 47(6):4436--4450, 2015.

\bibitem{Gamow1930}
G.~Gamow.
\newblock Mass defect curve and nuclear constitution.
\newblock {\em Proc. R. Soc. A}, 126(803):632--644, 1930.

\bibitem{Geiss1996}
C.~Geiß, F.~Heider, and H.~Soffel.
\newblock Magnetic domain observations on magnetite and titanornaghemite grains
  (0.5-10 µm).
\newblock {\em Geophys. J. lnt.}, 124:75--88, 1996.

\bibitem{Gradshteyn2000}
I.~Gradshteyn and I.~Ryzhik.
\newblock {\em Table of integrals, series, and products}.
\newblock Academic Press, Inc., San Diego, CA, sixth edition, 2000.

\bibitem{Grafakos2009}
L.~Grafakos.
\newblock {\em Classical Fourier Analysis}.
\newblock Springer New York, 2009.

\bibitem{Hess2015}
A.~J. Hess, Q.~Liu, and I.~I. Smalyukh.
\newblock Optical patterning of magnetic domains and defects in ferromagnetic
  liquid crystal colloids.
\newblock {\em Appl. Phys. Lett.}, 107(7):071906, aug 2015.

\bibitem{HS-Book}
A.~Hubert and R.~Schäfer.
\newblock {\em Magnetic Domains}.
\newblock Springer Berlin Heidelberg, 1998.

\bibitem{Julin2014}
V.~Julin.
\newblock Isoperimetric problem with a coulomb repulsive term.
\newblock {\em Indiana Univ. Math. J.}, 63(1):77--89, 2014.

\bibitem{Kimura2020}
M.~Kimura and P.~van Meurs.
\newblock Regularity of the minimiser of one-dimensional interaction energies.
\newblock {\em ESAIM: COCV}, 26:27, 2020.

\bibitem{King2009}
F.~W. King.
\newblock {\em Hilbert Transforms}, volume~1.
\newblock Cambridge Univ. Press, 2009.

\bibitem{KnuepferKohn-2010}
H.~Knüpfer and R.~V. Kohn.
\newblock Minimal energy for elastic inclusions.
\newblock {\em Proc. R. Soc. A}, 467(2127):695--717, 2010.

\bibitem{KnuepferMuratov-2011}
H.~Knüpfer and C.~Muratov.
\newblock Domain structure of bulk ferromagnetic crystals in applied fields
  near saturation.
\newblock {\em J. Nonlinear Sci.}, 21(6):921--962, 2011.

\bibitem{Knuepfer2013}
H.~Knüpfer and C.~Muratov.
\newblock On an isoperimetric problem with a competing nonlocal term i: The
  planar case.
\newblock {\em Commun. Pure Appl. Math.}, 66(7):1129--1162, 2013.

\bibitem{Knuepfer2012}
H.~Kn\"{u}pfer and C.~Muratov.
\newblock On an isoperimetric problem with a competing nonlocal term {II}:
  {T}he general case.
\newblock {\em Commun. Pure Appl. Math.}, 67(12):1974--1994, 2014.

\bibitem{Knupfer2017}
H.~Kn\"{u}pfer and F.~Nolte.
\newblock Optimal shape of isolated ferromagnetic domains.
\newblock {\em SIAM J. Math. Anal.}, 50(6):5857--5886, 2018.

\bibitem{KnuepferOtto-2019}
H.~Kn\"{u}pfer and F.~Otto.
\newblock Nucleation barriers for the cubic-to-tetragonal phase transformation
  in the absence of self-accommodation.
\newblock {\em ZAMM Z. Angew. Math. Mech.}, 99(2):e201800179, 12, 2019.

\bibitem{Leibler1980}
L.~Leibler.
\newblock Theory of microphase separation in block copolymers.
\newblock {\em Macromolecules}, 13(6):1602--1617, 1980.

\bibitem{Li2017}
Q.~Li, C.~W. Kartikowati, S.~Horie, T.~Ogi, T.~Iwaki, and K.~Okuyama.
\newblock Correlation between particle size/domain structure and magnetic
  properties of highly crystalline fe3o4 nanoparticles.
\newblock {\em Scientific Reports}, 7(1), aug 2017.

\bibitem{Lieb2001}
E.~H. Lieb and M.~Loss.
\newblock {\em Analysis (Graduate Studies in Mathematics)}.
\newblock American Mathematical Society, 2001.

\bibitem{Lu2013}
J.~Lu and F.~Otto.
\newblock Nonexistence of a minimizer for thomas-fermi-dirac-von weizsäcker
  model.
\newblock {\em Commun. Pure Appl. Math.}, 67(10):1605--1617, 2013.

\bibitem{Matsen1994}
M.~W. Matsen and M.~Schick.
\newblock Stable and unstable phases of a diblock copolymer melt.
\newblock {\em Phys. Rev. Lett.}, 72(16):2660--2663, 1994.

\bibitem{Mora2018}
G.~Mora, L.~Rondi, and L.~Scardia.
\newblock The equilibrium measure for a nonlocal dislocation energy.
\newblock {\em Communications on Pure and Applied Mathematics}, 72(1):136--158,
  jul 2018.

\bibitem{Murakami2009}
Y.~Murakami, H.~Kasai, J.~J. Kim, S.~Mamishin, D.~Shindo, S.~Mori, and
  A.~Tonomura.
\newblock Ferromagnetic domain nucleation and growth in colossal
  magnetoresistive manganite.
\newblock {\em Nature Nanotechnology}, 5(1):37--41, nov 2009.

\bibitem{Newbower1973}
R.~Newbower.
\newblock Magnetic fluids in the blood.
\newblock {\em IEEE Trans. Mag.}, 9(3):447--450, sep 1973.

\bibitem{Nezza2012}
E.~D. Nezza, G.~Palatucci, and E.~Valdinoci.
\newblock Hitchhiker's guide to the fractional sobolev spaces.
\newblock {\em Bulletin des Sciences Math{\'{e}}matiques}, 136(5):521--573, jul
  2012.

\bibitem{Nolte2018}
F.~Nolte.
\newblock {\em Optimal scaling laws for domain patterns in thin ferromagnetic
  films with strong perpendicular anisotropy}.
\newblock PhD thesis, Ruprecht Karl University of Heidelberg, 2018.

\bibitem{Ochonski1989}
W.~Ocho{\'{n}}ski.
\newblock Dynamic sealing with magnetic fluids.
\newblock {\em Wear}, 130(1):261--268, mar 1989.

\bibitem{Ohta1986}
T.~Ohta and K.~Kawasaki.
\newblock Equilibrium morphology of block copolymer melts.
\newblock {\em Macromolecules}, 19(10):2621--2632, oct 1986.

\bibitem{DLMF2021}
F.~W.~J. Olver, A.~B.~O. Daalhuis, D.~W. Lozier, B.~I. Schneider, R.~F.
  Boisvert, C.~W. Clark, B.~R. Miller, B.~V. Saunders, H.~S. Cohl, and
  e.~M.~A.~McClain.
\newblock Nist digital library of mathematical functions, 2021.

\bibitem{OttoViehmann-2009}
F.~Otto and T.~Viehmann.
\newblock Domain branching in uniaxial ferromagnets: asymptotic behavior of the
  energy.
\newblock {\em Calc. Var. Partial Differential Equations}, 38(1-2):135--181,
  2009.

\bibitem{Pegon2021}
M.~Pegon.
\newblock Large mass minimizers for an isoperimetric problem with a repulsive
  integrable potential.
\newblock {\em Nonlinear Analysis}, 211, 2021.

\bibitem{Polya1951}
G.~Pólya and G.~Szegő.
\newblock {\em Isoperimetric inequalities in mathematical physics}, volume~27.
\newblock Princeton University Press, 1951.

\bibitem{Pokhil1997}
T.~Pokhil and B.~Moskowitz.
\newblock Magnetic domains and domain walls in pseudo-single-domain magnetite
  studied with magnetic force microscopy.
\newblock {\em J. Geo. Res.}, 102(B10):22681--22694, 1997.

\bibitem{Prudnikov1990}
A.~P. Prudnikov.
\newblock {\em Integrals Series: More Special Functions}, volume~3.
\newblock CRC, 1990.

\bibitem{Roodan2020}
V.~A. Roodan, J.~G{\'{o}}mez-Pastora, I.~H. Karampelas,
  C.~Gonz{\'{a}}lez-Fern{\'{a}}ndez, E.~Bringas, I.~Ortiz, J.~J. Chalmers,
  E.~P. Furlani, and M.~T. Swihart.
\newblock Formation and manipulation of ferrofluid droplets with magnetic
  fields in a microdevice: a numerical parametric study.
\newblock {\em Soft Matter}, 16(41):9506--9518, 2020.

\bibitem{Rosensweig-Book}
R.~Rosensweig.
\newblock {\em Ferrohydrodynamics}, volume~1.
\newblock Cambridge Univ. Press, 1985.

\bibitem{Rueland2019}
A.~R\"{u}land.
\newblock Quantitative invertibility and approximation for the truncated
  {H}ilbert and {R}iesz transforms.
\newblock {\em Revista Matem\'{a}tica Iberoamericana}, 35(7):1997--2024, 2019.

\bibitem{SeroGuillaume1992}
O.~E. S{\'{e}}ro-Guillaume, D.~Zouaoui, D.~Bernardin, and J.~P. Brancher.
\newblock The shape of a magnetic liquid drop.
\newblock {\em J. Fluid Mech.}, 241:215--232, aug 1992.

\bibitem{Soehngen1939}
H.~S\"{o}hngen.
\newblock Die {L}\"{o}sungen einer {I}ntegralgleichung und deren {A}nwendung in
  der {T}ragfl\"{u}geltheorie.
\newblock {\em Mathematische Zeitschrift}, 45(1):245--264, 1939.

\bibitem{Stantejsky2018}
D.~Stantejsky.
\newblock A model problem in micromagnetics: Scaling, $\gamma-$limit and shape
  of minimizers.
\newblock Master's thesis, University of Heidelberg, 2018.

\bibitem{Strukov1998}
B.~A. Strukov and A.~P. Levanyuk.
\newblock {\em Ferroelectric Phenomena in Crystals}.
\newblock Springer Berlin Heidelberg, 1998.

\bibitem{Szczech2015}
M.~Szczech and W.~Horak.
\newblock Tightness testing of rotary ferromagnetic fluid seal working in water
  environment.
\newblock {\em Industrial Lubrication and Tribology}, 67(5):455--459, aug 2015.

\bibitem{Uhlmann2002}
E.~Uhlmann, G.~Spur, N.~Bayat, and R.~Patzwald.
\newblock Application of magnetic fluids in tribotechnical systems.
\newblock {\em Journal of Magnetism and Magnetic Materials}, 252:336--340, nov
  2002.

\bibitem{Meurs2021}
P.~van Meurs.
\newblock Expansions for the linear-elastic contribution to the
  self-interaction force of dislocation curves.
\newblock {\em arXiv}, 2021.

\end{thebibliography}

\end{document}